\documentclass[11pt,a4paper]{article}

\usepackage{amsmath,amssymb,amsthm,mathtools}
\usepackage{geometry}
\usepackage{enumitem}
\usepackage{booktabs,array,tabularx}
\usepackage{cite}
\usepackage[colorlinks=true,linkcolor=blue,citecolor=blue,urlcolor=blue]{hyperref}
\usepackage{url}

\hypersetup{
  pdftitle={Stein Structures Not Determined by the Contact Boundary},
  pdfauthor={Nobuo Iida},
  pdfsubject={Stein fillings, boundary contact structures, and Weinstein homotopy}
}
\usepackage{indentfirst}
\numberwithin{equation}{section}

\newtheorem{theorem}{Theorem}[section]
\newtheorem{proposition}[theorem]{Proposition}
\newtheorem{lemma}[theorem]{Lemma}
\newtheorem{corollary}[theorem]{Corollary}
\theoremstyle{definition}

\newtheorem{remark}[theorem]{Remark}
\newtheorem{question}[theorem]{Question}
\newtheorem{problem}[theorem]{Problem}
\newtheorem{example}[theorem]{Example}
\newtheorem*{unnumberedremark}{Remark}
\newtheorem*{unnumberedlemma}{Lemma}

\newcommand{\Z}{\mathbb{Z}}

\newcommand{\R}{\mathbb{R}}
\newcommand{\C}{\mathbb{C}}
\newcommand{\F}{\mathbb{F}_2}
\newcommand{\id}{\operatorname{id}}
\newcommand{\PD}{\operatorname{PD}}
\newcommand{\HM}{\widehat{HM}}
\newcommand{\ECH}{\operatorname{ECH}}
\newcommand{\SpinC}{\operatorname{Spin}^{c}}

\newcommand{\W}{\mathcal{W}}
\newcommand{\Int}{\operatorname{int}}
\newcommand{\norm}[1]{\lVert #1\rVert}
\newcommand{\Diff}{\operatorname{Diff}}

\title{Stein Structures Not Determined by the Contact Boundary}
\author{Nobuo Iida\thanks{Kavli Institute for the Physics and Mathematics of the Universe (Kavli IPMU), The University of Tokyo, 5-1-5 Kashiwanoha, Kashiwa, Chiba 277-8583, Japan. E-mail: \href{mailto:iidanobuo1224@g.ecc.u-tokyo.ac.jp}{\texttt{iidanobuo1224@g.ecc.u-tokyo.ac.jp}}; \href{mailto:iidanobuo1224@gmail.com}{\texttt{iidanobuo1224@gmail.com}}.}}
\date{}

\begin{document}

\maketitle

\begin{abstract}
We construct two Stein structures on the same compact smooth four-manifold
whose boundary contact forms are strictly identified and whose first Chern
classes agree, but whose exact symplectic forms are not symplectomorphic.
Moreover, no diffeomorphism makes the associated Weinstein structures
homotopic.

The examples are obtained by removing a tubular neighborhood of a smooth
bicanonical curve from a fake projective plane and comparing the induced Stein
structure with its conjugate. Their canonical $\SpinC$ structures differ by a
nonzero class of order two. After the boundary forms are normalized using a common Boothby--Wang connection form, a hypothetical symplectomorphism preserves 
the oriented circle fiber and extends over the divisor cap, contradicting
canonical-class rigidity. For the Weinstein statement, the required
fiber-preservation is obtained from a monopole Floer grading asymmetry, using
the Nelson--Weiler computation and Taubes's ECH--Seiberg--Witten
correspondence. This gives an affirmative solution to a normalized version of the relative
filling problem following Problem~4.96 in the $K3$ problem list.
\end{abstract}

\section{Introduction}

Stein manifolds lie at a meeting point of several complex variables, symplectic
topology, and smooth topology.  A strictly plurisubharmonic exhaustion equips a
Stein manifold with an exact symplectic structure and, after a standard Morse
perturbation, with a Weinstein structure.  In real dimension four, Stein
surfaces also admit concrete handlebody descriptions in terms of Legendrian
attachments; see \cite{GompfStein,CieliebakEliashberg}.  Their strictly
pseudoconvex boundaries carry canonical cooriented contact structures, so a
compact Stein surface is naturally a symplectic and Weinstein filling of a
contact $3$--manifold.

A basic rigidity question is how much of a Stein or Weinstein filling is
remembered by its contact boundary.  The question studied here asks for a
particularly subtle form of nonuniqueness: can two Stein structures live on
diffeomorphic $4$--manifolds, have contactomorphic cooriented boundaries and
matching first Chern classes, while remaining inequivalent as symplectic or
Weinstein structures even after arbitrary reparametrization?

This question is related to Problem~4.96 in the K3 problem list; see \cite[Problem~4.96]{K3ProblemList}.

\begin{question}[K3 Problem 4.96]
\label{ques:k3printed}
Are there symplectic $4$--manifolds $(X_1,\omega_1)$ and $(X_2,\omega_2)$,
together with a diffeomorphism $f:X_1\to X_2$, for which
\[
 f^*c_1(X_2,\omega_2)=c_1(X_1,\omega_1),
 \qquad
 f^*[\omega_2]=[\omega_1]\in H^2(X_1;\mathbb R),
\]
although the two symplectic manifolds are not symplectomorphic?
\end{question}

The adjective ``closed'' does not occur in the printed question.  Its position
in the discussion of symplectic structures on closed four-manifolds, followed
immediately by a separate relative question for fillings, makes the following
closed interpretation natural.  We distinguish the printed wording from that
interpretation.

\begin{question}[The closed interpretation of K3 Problem 4.96]
\label{ques:k3closed}
Does there exist a pair of closed symplectic $4$--manifolds
$(X_1,\omega_1)$ and $(X_2,\omega_2)$ and a diffeomorphism
$f:X_1\to X_2$ such that
\[
 f^*c_1(X_2,\omega_2)=c_1(X_1,\omega_1),
 \qquad
 f^*[\omega_2]=[\omega_1]\in H^2(X_1;\R),
\]
but $(X_1,\omega_1)$ and $(X_2,\omega_2)$ are not symplectomorphic?
\end{question}

We do not answer Question~\ref{ques:k3closed}.  If manifolds with boundary are
admitted in the printed wording, Theorem~\ref{thm:main} gives a
boundary-bearing example: its two forms are exact, the identity matches their
first Chern classes, and they are not symplectomorphic.  We record this only as
a convention-dependent observation, not as a solution of the closed
interpretation stated above.

For comparison, let $M$ be a closed oriented smooth manifold, let
$a\in H^2(M;\R)$, and set
\[
 \mathcal S_a
 :=\{\rho\in\Omega^2(M)\mid \rho\text{ is symplectic and }[\rho]=a\}.
\]
Salamon asks the following \cite[Question~2]{SalamonUniqueness}.

\begin{question}[Salamon's Question 2]
\label{ques:salamon}
Are any two symplectic forms in $\mathcal S_a$ diffeomorphic?
\end{question}

Here ``diffeomorphic'' means that for every
$\omega_0,\omega_1\in\mathcal S_a$ there is a diffeomorphism $g:M\to M$
with $g^*\omega_1=\omega_0$.

\begin{proposition}[Equivalence of the two closed questions]
\label{prop:k3-salamon}
Question~\ref{ques:k3closed} has an affirmative answer if and only if
Question~\ref{ques:salamon} has a negative answer for some closed smooth
four-manifold $M$ and some class $a\in H^2(M;\R)$ with
$\mathcal S_a\neq\varnothing$.
\end{proposition}

\begin{proof}
Suppose first that Question~\ref{ques:k3closed} has an affirmative answer.
On $M=X_1$ put $\widetilde\omega_2=f^*\omega_2$.  Then
$[\widetilde\omega_2]=[\omega_1]$.  If the two forms were diffeomorphic, there
would be a diffeomorphism $g:X_1\to X_1$ satisfying
$g^*\widetilde\omega_2=\omega_1$.  The composition $f\circ g$ would then be a
symplectomorphism from $(X_1,\omega_1)$ to $(X_2,\omega_2)$, a contradiction.
Thus Question~\ref{ques:salamon} has a negative answer for $a=[\omega_1]$.

Conversely, suppose that $\omega_0,\omega_1\in\mathcal S_a$ are not related by
any diffeomorphism of $M$.  Take $X_1=X_2=M$ and use the identity as the
comparison diffeomorphism.  Salamon's Corollary~A implies that two
cohomologous symplectic forms on a closed smooth four-manifold determine
equivalent canonical $\SpinC$ structures and hence the same first Chern class
\cite[Corollary~A]{SalamonUniqueness}.  Therefore all the hypotheses of
Question~\ref{ques:k3closed} hold, while the two symplectic manifolds are not
symplectomorphic.
\end{proof}

Thus the first-Chern-class condition in Question~\ref{ques:k3closed} is redundant:
a counterexample is precisely a closed smooth four-manifold $M$, a class
$a\in H^2(M;\R)$, and forms $\omega_0,\omega_1\in\mathcal S_a$ such that
\[
 \omega_0\neq g^*\omega_1
 \qquad\text{for every }g\in\Diff(M).
\]
To the author's knowledge, as of August~2026, no such closed
four-dimensional example is known.
Here one must distinguish path-connectedness of the fixed-class space
$\mathcal S_a$ from membership in a single $\Diff^+(M)$-orbit.  Lin--Wu
construct infinitely many pairwise non-isotopic, cohomologous symplectic forms
on ruled surfaces, but their forms are mutually diffeomorphic
\cite{LinWuComponents}.  Li--Ning prove partial uniqueness results after
restricting to K\"ahler-type forms and emphasize that the full quotient by
diffeomorphisms remains poorly understood
\cite[Sections~1.2--1.4]{LiNingSpaces}.

The problem list then says, ``There is also a relative version of this problem
for fillings'' \cite[Problem~4.96]{K3ProblemList}, and asks:

\begin{question}[Relative question following Problem~4.96]
\label{ques:k3relative}
Does there exist a pair of Stein manifolds filling the same contact
3-manifold that are diffeomorphic where the diffeomorphism takes the first
Chern class of one to that of the other, but such that the Stein manifolds are
not symplectomorphic?
\end{question}

One could also ask whether there is such a pair that is not Weinstein
homotopic; see also Problem~5.17.  In this paper, ``filling the same contact
$3$--manifold'' means that the two cooriented boundary contact manifolds are
contactomorphic by some contactomorphism that is not specified in advance.

Taken literally as a question about selected exact K\"ahler forms on compact
Stein domains, the symplectic clause has an elementary answer.

\begin{example}[The standard ball with two scales]\label{ex:scaling}
Let $B^4\subset\C^2$ carry the standard complex structure and put
\[
 \rho_a=a\bigl(|z|^2-1\bigr),\qquad
 \lambda_a=d^c\rho_a,\qquad
 \omega_a=d\lambda_a,
 \qquad a>0.
\]
For every $a$, the identity identifies the boundary contact structure with
$(S^3,\xi_{\mathrm{st}})$ and matches the first Chern classes, since
$c_1(TB^4)=0$.  Moreover,
\[
 \lambda_a=a\lambda_1,
 \qquad
 \omega_a=a\omega_1,
 \qquad
 \int_{B^4}\omega_a^2=a^2\int_{B^4}\omega_1^2.
\]
Thus $(B^4,\omega_a)$ and $(B^4,\omega_1)$ are not symplectomorphic for
$a\ne1$.  The path
\[
 \rho_t=((1-t)+ta)(|z|^2-1),\qquad 0\le t\le1,
\]
is nevertheless a Weinstein homotopy, and the two completions are
symplectomorphic by radial dilation.  This example records only the scale of a
chosen compact exact form.
\end{example}

The substantive relative question is the following normalized version:

\begin{problem}[Main problem]
\label{prob:substantive}
Determine whether the following data exist: compact connected Stein surfaces
$(X_i,J_i,\rho_i)$, a diffeomorphism $f:X_1\to X_2$, and a boundary
diffeomorphism $\tau:\partial X_1\to\partial X_2$.  With
\[
 \lambda_i=d^c_{J_i}\rho_i,\qquad \omega_i=d\lambda_i,
\]
one has
\[
 f^*c_1(TX_2,J_2)=c_1(TX_1,J_1),
 \qquad
 \tau^*(\lambda_2|_{\partial X_2})=\lambda_1|_{\partial X_1},
\]
but $(X_1,\omega_1)$ and $(X_2,\omega_2)$ are not symplectomorphic.  Can one
moreover require that no diffeomorphism makes the associated Weinstein
structures Weinstein homotopic?  (No compatibility between $\tau$ and
$f|_{\partial X_1}$ is imposed.)
\end{problem}

The equality of boundary forms is stronger than merely identifying their
cooriented kernels as contact plane fields.  It is imposed to eliminate the
scaling phenomenon in Example~\ref{ex:scaling}: by Stokes' theorem it forces
equality of the exact symplectic volumes. Once a boundary contactomorphism is chosen, Lemma~\ref{lem:definingrescale}
modifies the strictly plurisubharmonic defining functions so that the selected
boundary forms agree strictly.  The map $\tau$ is not required to be the
boundary restriction of $f$; indeed, in our example no boundary
contactomorphism is isotopic to the boundary restriction of a diffeomorphism of
the filling; see Remark~\ref{rem:marking-not-extend}.

\begin{theorem}[Main theorem]\label{thm:main}
There exist a compact connected smooth four-manifold $X$, two Stein structures
$J_+$ and $J_-$ on $X$, strictly plurisubharmonic defining functions
$\rho_+$ and $\rho_-$, and a diffeomorphism $\tau:Y\to Y$, where
$Y=\partial X$ is the common underlying boundary.  Writing
\[
 \lambda_\pm=d^c_{J_\pm}\rho_\pm,
 \qquad
 \omega_\pm=d\lambda_\pm,
 \qquad
 \xi_\pm=\ker(\lambda_\pm|_Y),
\]
one has:
\begin{enumerate}[label=\textup{(\roman*)}]
\item $\tau$ is a strict coorientation-preserving contactomorphism:
\[
 \tau^*(\lambda_-|_Y)=\lambda_+|_Y;
\]
\item the first Chern classes agree:
\[
 c_1(TX,J_+)=c_1(TX,J_-)
 \quad\text{in }H^2(X;\Z);
\]
\item the canonical $\SpinC$ structures are nevertheless distinct; more
precisely, if $k_X=K|_X$, then $k_X$ has order two and
\[
 \mathfrak{s}_{J_-}=\mathfrak{s}_{J_+}+k_X
 \neq\mathfrak{s}_{J_+};
\]
\item the compact symplectic manifolds with boundary are not symplectomorphic:
\[
 (X,\omega_+)\not\cong_{\mathrm{symp}}(X,\omega_-);
\]
\item for every diffeomorphism $F:X\to X$, the Weinstein structure associated
with $(J_+,\rho_+)$ is not Weinstein homotopic to the pullback by $F$ of the
Weinstein structure associated with $(J_-,\rho_-)$.
\end{enumerate}
\end{theorem}
The strict equality in part~\textup{(i)} is obtained by changing the strictly
plurisubharmonic defining functions while keeping $J_\pm$ fixed.  It is a
normalization of the selected contact forms, as opposed to merely an equality
of their kernels under a contactomorphism.  It implies
\[
 \int_X\omega_+^2=\int_X\omega_-^2,
\]
so the failure of symplectomorphism is not caused by scaling or a volume
mismatch.  Moreover, both $(Y,\xi_+)$ and $(Y,\xi_-)$ are contactomorphic to the
prequantization contact structure on the oriented circle bundle of Euler
number $-36$ over a closed oriented surface of genus $28$.

Here an ``associated Weinstein structure'' means a representative obtained
from the indicated Stein datum by a sufficiently small Morse perturbation
supported in the interior.  The strict boundary map in part~\textup{(i)} is
unmarked: it is not isotopic to the boundary restriction of any diffeomorphism
of $X$; see Remark~\ref{rem:marking-not-extend}.

\begin{corollary}[Solution of the Main problem]
\label{cor:k3}
The pair in Theorem~\ref{thm:main} satisfies the Main problem, including its
non-Weinstein-homotopic clause.
\end{corollary}

Note that, by a result of Lisca and Mati\'c, one can see from
part~\textup{(iii)} that $\xi_+$ and $\xi_-$ are not isotopic as cooriented
contact structures.  We recall their theorem in the form used here.

\begin{theorem}[Lisca--Mati\'c]\label{thm:lisca-matic}
Let $W$ be a compact smooth four-manifold with boundary, and let $J_0,J_1$ be
Stein structures on $W$ with associated $\SpinC$ structures
$\mathfrak{s}_{J_0},\mathfrak{s}_{J_1}$.  If the induced cooriented contact
structures on $\partial W$ are isotopic, then
\[
 \mathfrak{s}_{J_0}\cong\mathfrak{s}_{J_1}.
\]
\end{theorem}

This is \cite[Theorem~1.2]{LiscaMatic}; its proof uses an embedding method and four-dimensional Seiberg--Witten theory.

\begin{corollary}[Boundary nonisotopy from Lisca--Mati\'c]
\label{cor:boundary-nonisotopy}
On the common underlying boundary $Y=\partial X$, the contact structures
$\xi_+$ and $\xi_-$ in Theorem~\ref{thm:main} are not isotopic as
cooriented contact structures.
\end{corollary}

\begin{proof}
If $\xi_+$ and $\xi_-$ were isotopic as cooriented contact structures,
Theorem~\ref{thm:lisca-matic} would imply
$\mathfrak{s}_{J_+}\cong\mathfrak{s}_{J_-}$, contradicting
part~\textup{(iii)} of Theorem~\ref{thm:main}.
\end{proof}

\subsection*{Other dimensions}

The Main problem has a simple negative answer in real dimension two and known
affirmative solutions in an infinite family of higher dimensions.

\begin{proposition}[Real dimension two]
\label{prop:dimension-two}
For compact connected Stein domains, the Main problem has no solution in real
dimension two.
\end{proposition}

\begin{proof}
Let $(X_i,J_i,\rho_i)$ be compact connected Stein curves satisfying the smooth
and boundary-form requirements of the Main problem, and put
$\lambda_i=d^c_{J_i}\rho_i$ and $\omega_i=d\lambda_i$.  The first-Chern-class
condition is automatic because $H^2(X_i;\mathbb Z)=0$.  Since
$\lambda_i|_{\partial X_i}$ is positive for the boundary orientation, the
strict boundary equality implies
\[
 \int_{X_1}\omega_1
 =\int_{\partial X_1}\lambda_1
 =\int_{\partial X_1}\tau^*\lambda_2
 =\int_{\partial X_2}\lambda_2
 =\int_{X_2}\omega_2.
\]
Choose an orientation-preserving diffeomorphism $g:X_1\to X_2$; such a map
exists by the classification of compact oriented surfaces with boundary.
Then $\omega_1$ and $g^*\omega_2$ are positive volume forms on $X_1$ with the
same total volume.  Banyaga's extension of Moser's theorem to manifolds with
boundary \cite{BanyagaVolumeBoundary} gives a diffeomorphism
$\phi:X_1\to X_1$ such that
\[
 \phi^*g^*\omega_2=\omega_1.
\]
Thus $g\circ\phi:(X_1,\omega_1)\to(X_2,\omega_2)$ is a symplectomorphism, so
the non-symplectomorphism requirement can never hold.
\end{proof}

The higher-dimensional Main problem is obtained from
Problem~\ref{prob:substantive} by allowing compact connected Stein domains of
arbitrary complex dimension and contact boundaries of the corresponding odd
dimension.

\begin{proposition}[Known higher-dimensional solutions]
\label{prop:higher-main}
For every $k\geq2$, the higher-dimensional Main problem has an affirmative
solution in real dimension $4k+2$.
\end{proposition}

\begin{proof}[Proof sketch]
For every even complex dimension $m\geq4$, Kartal constructs diffeomorphic
Weinstein domains of complex dimension $m+1$ with the same contact boundary
and inequivalent wrapped Fukaya categories
\cite[Theorem~1.1, Examples~1.3--1.4, and Corollary~1.6]{KartalMappingTori}.
After the common subcritical handle attachments in his Corollary~1.6, the
domains are simply connected and their canonical bundles have a unique
homotopy class of trivialization; in particular their first Chern classes
vanish.  The Stein--Weinstein correspondence gives Stein representatives on
the two diffeomorphic domains.  The dimension-independent rescaling lemma,
Lemma~\ref{lem:definingrescale}, then changes the strictly plurisubharmonic
defining functions so that the selected boundary contact forms agree under a
chosen boundary contactomorphism.

These changes are Weinstein deformations and therefore do not remove the
wrapped-Fukaya-category obstruction.  Since the domains are simply connected,
any symplectomorphism between their exact symplectic forms would be exact: the
difference between the pulled-back and original primitives is a closed
one-form and hence exact.  Thus no symplectomorphism exists, and a Weinstein
homotopy after any diffeomorphic identification would likewise contradict the
wrapped Fukaya categories.  Since $m+1$ is odd and at least five, the resulting
real dimensions are precisely $4k+2\geq10$.
\end{proof}

Together with Theorem~\ref{thm:main}, these statements settle the Main problem
in real dimensions $2$, $4$, and $4k+2\geq10$.  The remaining real dimensions
$6$ and $4k$ with $k\geq2$ are recorded as questions in
Section~\ref{sec:further}.

\subsection*{Organization}

Section~\ref{sec:conventions} fixes some conventions.  Section~\ref{sec:algebraic} chooses the fake projective
plane and a smooth bicanonical curve.  Section~\ref{sec:stein} constructs the
conjugate initial Stein data and computes their first Chern and canonical
$\SpinC$ structures.  Section~\ref{sec:boundary} identifies the boundary
prequantization spaces and fixes all fiber-orientation signs.
Section~\ref{sec:normalization} constructs the final normalized defining
functions and realizes both selected boundary forms as pullbacks of one
Boothby--Wang connection form.  Section~\ref{sec:cap} first records the
topology of the underlying nontrivial circle bundle, then states the
Waldhausen and Chen--Tshishiku mapping-class inputs, proves the cap-extension
proposition, and proves part~\textup{(iv)} without Floer homology.
Section~\ref{sec:monopole} states the Nelson--Weiler--Taubes grading formula
and proves detection of the positive oriented fiber class.
Section~\ref{sec:weinstein} uses that Floer-theoretic detection to prove
part~\textup{(v)} and then assembles the proofs of the Main Theorem and its
corollary.  Finally, Section~\ref{sec:further} records further remarks, the
marked version of the problem, the question of a simply connected
four-dimensional filling, and the remaining real dimensions $6$ and $4k$
with $k\geq2$.

\section{Conventions and preliminaries}\label{sec:conventions}

All manifolds are oriented.  A contact structure is cooriented, and a contact
isotopy preserves the coorientation.  
Let $(X,J)$ be a compact
complex manifold with boundary. 
 Our convention is 
\[
 d_J^cf=-df\circ J
\]
 and hence
\begin{equation}\label{eq:ddcconvention}
d^c_J=\sqrt{-1}(\bar\partial-\partial), \quad dd^c_J=2\sqrt{-1}\,\partial\bar\partial.
\end{equation}
A smooth function $\rho: X\to \R$ is said to be strictly $J$-plurisubharmonic if  the
Levi form $dd_J^c\rho(v,Jv)$ is positive for every nonzero tangent vector
$v$.
A triple $(X,J, \rho)$  is  defined to be a compact Stein domain if $(X, J)$ is a compact
complex manifold with boundary together with a smooth strictly
$J$-plurisubharmonic defining function $\rho\leq0$ such that
\[
 \partial X=\rho^{-1}(0),\qquad d\rho|_{\partial X}\neq0.
\]
A Stein datum $(J,\rho)$ determines the Liouville form
$\lambda=d_J^c\rho$, the exact symplectic form $\omega=d\lambda$, and the
positive cooriented boundary contact structure
\[
 \xi=\ker\alpha, \quad \alpha=\lambda|_{\partial X}
\]
whose coorientation is represented by the contact form
$\alpha$. 
If $\phi$ is strictly $J$-plurisubharmonic, then
\[
 \lambda=d_J^c\phi,
 \qquad
 \omega=dd_J^c\phi,
 \qquad
 g(v,w)=\omega(v,Jw).
\]
The Liouville vector field is the $g$-gradient of $\phi$.  We orient the boundary
of a regular sublevel set by the outward-normal-first convention.  This is the
Liouville boundary orientation and satisfies
\[
 \alpha \wedge d\alpha >0.
\]
With these conventions, Stokes' theorem reads
\begin{equation}\label{eq:stokes-volume}
 \int_X(d\lambda)^2=\int_{\partial X}\lambda\wedge d\lambda.
\end{equation}

A symplectomorphism between two symplectic manifolds with boundary is a
diffeomorphism $F:(X_1,\omega_1)\to(X_2,\omega_2)$ satisfying
$F^*\omega_2=\omega_1$.  Unless explicitly stated otherwise, we impose no
boundary marking, exactness condition, or requirement that $F$ preserve a
chosen primitive.
\par
Next, let us define a Weinstein structure.
Let $X$ be a compact smooth manifold with boundary.
A smooth $1$-form $\lambda$ on $X$ is called a \emph{Liouville form} if
\[
d\lambda
\]
is a symplectic form.
The associated Liouville vector field $Z$ is uniquely determined by
\[
\iota_Z d\lambda=\lambda.
\]

A smooth function
\[
\varphi\colon X\to\mathbb{R}
\]
is called a generalized Morse function if all its critical points are
either Morse critical points or birth--death points.
At a birth--death point, after a smooth change of coordinates and the addition
of a constant, $\varphi$ has the local form
\[
x_1^3\pm x_2^2\pm\cdots\pm x_n^2.
\]
In a generic one-parameter family of generalized Morse functions, the
parameters at which birth--death points occur are isolated.

Let $V$ be a smooth vector field on $X$.
We say that $V$ is gradient-like for $\varphi$ if there exist a
Riemannian metric $g$ on $X$ and a positive smooth function
\[
\delta\colon X\to(0,\infty)
\]
such that
\[
d\varphi(V)
\geq
\delta
\left(
| V|_g^2
+
|d\varphi|_{g^*}^2
\right)
\]
at every point of $X$, where $g^*$ denotes the dual metric on $T^*X$.
In particular, this condition implies
\[
\operatorname{Zero}(V)=\operatorname{Crit}(\varphi)
\]
and
\[
d\varphi(V)>0
\]
away from the critical points of $\varphi$.

A Weinstein structure on $X$ is a triple
\[
\mathcal{W}=(X,\lambda,\varphi)
\]
such that the Liouville vector field $Z$ associated with the Liouville form $\lambda$ is outward
transverse to $\partial X$, the function $\varphi$ is a generalized Morse
function having $\partial X$ as a regular maximal level set, and $Z$ is
gradient-like for $\varphi$.

If the strictly plurisubharmonic defining function of a Stein datum is Morse, then
$(X,d_J^c\rho,\rho)$ is Weinstein: with respect to the K\"ahler metric its
Liouville vector field is the gradient of $\rho$.
After an arbitrarily small strictly plurisubharmonic Morse perturbation
supported in the interior,\footnote{By a strictly plurisubharmonic Morse
perturbation of $\rho$ we mean a function
\[
\rho'=\rho+\eta,
\]
where $\eta$ is a sufficiently small smooth function supported in
$\operatorname{int}X$, such that $\rho'$ is Morse and remains strictly
plurisubharmonic.  Such perturbations exist because the Morse functions are
dense in the space of smooth functions, while strict plurisubharmonicity is an
open condition in the $C^2$ topology.  Since $\eta$ is supported in the
interior, $\rho'$ agrees with $\rho$ near $\partial X$ and hence has the same
regular boundary level set.}
every compact Stein domain therefore determines a Weinstein structure.

Throughout, a Weinstein homotopy on a compact domain means a smooth family
$(\lambda_t,\varphi_t)$ on one fixed smooth manifold $X$ such that
$d\lambda_t$ is symplectic, the Liouville vector field is outward transverse to
$\partial X$, and $\partial X$ remains a regular convex level of
$\varphi_t$ for every $t$.  We use the standard generalized-Morse convention,
allowing isolated birth--death times; an arbitrarily small parametric
perturbation supported in the interior makes the functions Morse away from
those times without changing the boundary germ.  A Stein homotopy is a smooth family $(J_t,\rho_t)$ of Stein data on
$X$ for which $0$ is a regular value of every $\rho_t$.
It induces a Liouville homotopy via $\lambda_t=d^c_{J_t}\rho_t$ and, after a
parametric interior Morse perturbation, a Weinstein homotopy.  Whenever a
Weinstein representative of a Stein datum is needed, we choose an arbitrary
sufficiently small Morse perturbation supported away from the boundary.  No
uniqueness assertion for that auxiliary choice is needed below, because
Theorem~\ref{thm:nonweinstein} is formulated for every Weinstein structure
inducing the prescribed boundary contact structure.  Two Weinstein structures
on diffeomorphic domains are Weinstein deformation equivalent if, after
pullback by a diffeomorphism, they are connected by such a homotopy; see
\cite[Chapters~11--12]{CieliebakEliashberg}.

For a complex surface $S$, we write $K_S$ for the canonical line bundle and,
when no confusion can arise, use the same symbol
\[
 K=c_1(K_S)\in H^2(S;\Z)
\]
for its first Chern class (and for the corresponding divisor class).  Thus
\[
 c_1(TS,J)=-K.
\]
\section{A fake projective plane and a bicanonical curve}
\label{sec:algebraic}
A fake projective plane is a smooth complex projective surface with the same
rational homology as $\mathbb{CP}^2$ but not biholomorphic to
$\mathbb{CP}^2$.  The first example was constructed
by Mumford in 1979 using $p$-adic uniformization.  The subsequent classification of
Prasad--Yeung and Cartwright--Steger is complete; in particular, there are
precisely 100 fake projective planes up to biholomorphism; see \cite{PrasadYeung}.

Every fake projective plane is minimal of general type and satisfies equality
in the Bogomolov--Miyaoka--Yau inequality.  More explicitly,
\[
 c_2(S)=3,
 \qquad
 K_S^2=9=3c_2(S).
\]
By the equality case of the Bogomolov--Miyaoka--Yau theorem, its universal
cover is biholomorphic to the complex two-ball
\[
 \mathbb B_{\mathbb C}^{2}
 =
 \left\{
 (z_1,z_2)\in\mathbb C^2
 \;\middle|\;
 |z_1|^2+|z_2|^2<1
 \right\}.
\]
Consequently every fake projective plane can be written as a compact complex
ball quotient
\[
 S=\mathbb B_{\mathbb C}^{2}/\Gamma,
\]
where
\[
 \Gamma<\mathrm{PU}(2,1)
 =
 U(2,1)/Z(U(2,1)),
 \qquad
 Z(U(2,1))
 =
 \{e^{\sqrt{-1}\theta}I\mid \theta\in\mathbb R\}
 \cong U(1),
\]
is a torsion-free cocompact lattice;\footnote{A lattice in a locally compact
group $G$ is a discrete subgroup $\Gamma<G$ such that $G/\Gamma$ has finite
invariant volume.  It is called \emph{cocompact} (or \emph{uniform}) when
$G/\Gamma$ is compact.  In the present situation, discreteness implies that
the action of $\Gamma$ on $\mathbb B_{\mathbb C}^2$ is properly
discontinuous.  If an element of $\Gamma$ fixes a point of
$\mathbb B_{\mathbb C}^2$, it belongs to a compact point stabilizer; since
$\Gamma$ is discrete, the intersection with such a stabilizer is finite.
Thus torsion-freeness implies that the action is free.}
see \cite[Section~1.1]{PrasadYeung}.  Notice that no particular arithmetic
presentation of $\Gamma$ will be needed below: we only use the geometric and
rigidity properties shared by all fake projective planes.

\begin{proposition}[Kharlamov--Kulikov]\label{prop:KKsurface}
Any fake projective plane $S$ has the following properties:
\begin{enumerate}[label=\textup{(\roman*)}]
\item $K_S$ is ample\footnote{Let $L$ be a holomorphic line bundle over a compact complex manifold $M$.  The line bundle $L$ is \emph{ample} if there exist an integer $m>0$ and sections $s_0,\ldots,s_N\in H^0(M,L^{\otimes m})$ with no common zero such that \[  x\longmapsto[s_0(x):\cdots:s_N(x)] \] is a holomorphic embedding of $M$ into $\mathbb{CP}^N$.  If $h$ is a Hermitian metric on $L$ and a local holomorphic frame $e$ satisfies $\lVert e\rVert_h^2=e^{-\psi}$, the real Chern curvature form in our convention is \[  \Omega_h  =\sqrt{-1}\,\partial\bar\partial\psi  =\frac12dd^c\psi. \] The metric $h$ has \emph{positive Chern curvature} if \[  \Omega_h(v,Jv)>0 \] for every nonzero real tangent vector $v$.  The holomorphic line bundle $L$ is called \emph{positive} if it admits such a Hermitian metric.  For a holomorphic line bundle over an arbitrary compact complex manifold, positivity is equivalent to ampleness by the Kodaira embedding theorem. In the present paper, ampleness of $K_S$ is used both to invoke Reider's theorem and to choose a Hermitian metric with positive Chern curvature.};
\item
\[
 K^2=9;
\]
\item every self-homeomorphism $G:S\to S$ satisfies
\[
 G^*K=K
 \qquad\text{in }H^2(S;\Z),
\]
where $K=c_1(K_S)$ is the integral canonical class.
\end{enumerate}
\end{proposition}

\begin{proof}

Since $S$ has the rational homology of $\mathbb{CP}^2$, one has
\[
 b_1(S)=0,
 \qquad
 b_2(S)=1,
 \qquad
 c_2(S)=e(S)=3.
\]
Because $S$ is K\"ahler, $b_1=2q$ gives $q=0$, and the Hodge decomposition
together with $b_2=1$ gives $p_g=0$.  Hence
\[
 \chi(\mathcal O_S)=1-q+p_g=1.
\]
Noether's formula then gives
\[
 K_S^2
 =
 12\chi(\mathcal O_S)-c_2(S)
 =
 12-3
 =
 9.
\]
Equivalently,
\[
 K_S^2=3c_2(S),
\]
which is the equality case of the Bogomolov--Miyaoka--Yau inequality.  The
complex hyperbolic metric on the ball quotient has negative Ricci curvature,
so the canonical bundle is positive; in particular $K_S$ is ample.

It remains to explain the rigidity statement in part~\textup{(iii)}.
Write
\[
 S=\mathbb B_{\mathbb C}^2/\Gamma,
 \qquad
 \Gamma=\pi_1(S).
\]
Since $\mathbb B_{\mathbb C}^2$ is contractible, $S$ is a
$K(\Gamma,1)$.  Thus a self-homeomorphism
\[
 G:S\longrightarrow S
\]
is determined up to homotopy by the induced outer automorphism
\[
 [G_*]\in\operatorname{Out}(\Gamma).
\]
Strong rigidity for compact complex ball quotients implies that this outer
automorphism is realized by either a holomorphic or an anti-holomorphic
diffeomorphism
\[
 g:S\longrightarrow S.
\]
Equivalently, $G$ is homotopic to such a diffeomorphism $g$.  This is the
rigidity input used in
\cite[Proposition~2.1]{KharlamovKulikov}; see also Siu's strong rigidity
theorem for compact K\"ahler manifolds.

We next explain why the anti-holomorphic alternative cannot occur for a fake
projective plane.  Kharlamov and Kulikov prove the stronger statement that a
fake projective plane admits no anti-holomorphic diffeomorphism
\cite[Theorem~5.1]{KharlamovKulikovReal}.  The first step is to exclude an
anti-holomorphic involution
\[
 c:S\longrightarrow S,
 \qquad c^2=\id.
\]
Let
\[
 S_{\mathbb R}=\operatorname{Fix}(c).
\]
Since $S$ has the rational cohomology of $\mathbb{CP}^2$, one has
\[
 H^0(S;\mathbb Q)\cong H^2(S;\mathbb Q)\cong H^4(S;\mathbb Q)\cong\mathbb Q,
 \qquad
 H^1(S;\mathbb Q)=H^3(S;\mathbb Q)=0.
\]
The involution $c$ acts as $+1$ on $H^0(S;\mathbb Q)$.  It acts as $-1$
on $H^2(S;\mathbb Q)$, since this group is one-dimensional and generated
over $\mathbb Q$ by the nonzero canonical class $K$, while
$c^*K=-K$ for an anti-holomorphic map.  Finally, $c$ acts as $+1$ on
$H^4(S;\mathbb Q)$, because an anti-holomorphic map in complex dimension
two preserves the underlying real orientation.  Hence its Lefschetz number is
\[
 L(c)
 =
 \sum_{i=0}^{4}(-1)^i
 \operatorname{tr}\!\left(c^*:H^i(S;\mathbb Q)\to H^i(S;\mathbb Q)\right)
 =
 1-1+1
 =
 1.
\]
For an anti-holomorphic involution the fixed locus
$S_{\mathbb R}=\operatorname{Fix}(c)$ is a smooth real surface, and along
each fixed component the differential is $+1$ on the tangent bundle and
$-1$ on its normal bundle\footnote{More generally, the fixed-point set of a smooth involution is a
smooth submanifold.  To see this, choose a Riemannian metric invariant under
the involution $c$ by averaging an arbitrary metric over the $\mathbb Z/2$
action.  At a fixed point $p$, the differential $dc_p$ is then an orthogonal
involution, so
\[
 T_pS=E_+(p)\oplus E_-(p),
 \qquad
 E_\pm(p)=\ker(dc_p\mp\id).
\]
The exponential map of the invariant metric is $c$-equivariant, so in a sufficiently small neighborhood of $p$, the fixed locus is
$\exp_p(E_+(p))$.  In particular,
\[
 T_p\operatorname{Fix}(c)=E_+(p).
\]
Since $dc_p$ is an orthogonal involution, $E_-(p)$ is the orthogonal
complement of $E_+(p)$ and therefore identifies with the normal space to the
fixed locus.  Thus $dc_p$ acts as $+\id$ on the tangent space to each fixed
component and as $-\id$ on its normal space.

If $c$ is anti-holomorphic, then
\[
 dc_p\circ J=-J\circ dc_p.
\]
Consequently $J$ interchanges the two eigenspaces:
\[
 J E_+(p)=E_-(p),
 \qquad
 J E_-(p)=E_+(p).
\]
They therefore have the same real dimension, namely
$\dim_{\mathbb C}S$.  Hence, when $S$ is a complex surface, every nonempty
component of $\operatorname{Fix}(c)$ is a smooth real surface.  Moreover,
$J(T_p\operatorname{Fix}(c))$ is its normal space, so the fixed locus is
totally real.}.  Thus the clean Lefschetz fixed-point formula
identifies the contribution of each fixed component with its Euler
characteristic, and therefore
\[
 L(c)=\chi(S_{\mathbb R}).\footnote{One way to see this is to perturb the clean fixed locus to isolated
fixed points.  Let $F$ be a connected component of
$S_{\mathbb R}=\operatorname{Fix}(c)$ and choose a $c$-invariant tubular
neighborhood, identified with the normal bundle $NF$.  In these coordinates
the involution has the form
\[
 c(x,v)=(x,-v).
\]
Choose a Morse function $h:F\to\mathbb R$, let $\psi_t$ be the time-$t$
flow of $-\nabla h$, and extend $\psi_t$ to an isotopy $\varphi_t$ of a
neighborhood of $F$ which acts trivially in the normal direction and is
supported in that neighborhood.  Replacing $c$ by the homotopic map
$c_t=\varphi_t\circ c$, one has locally
\[
 c_t(x,v)=(\psi_t(x),-v).
\]
For sufficiently small $t>0$, the fixed points of $c_t$ near $F$ are
precisely $(p,0)$ with $p\in\operatorname{Crit}(h)$, and they are
nondegenerate.  At such a point,
\[
 dc_t=d\psi_t\oplus(-\id),
\]
so the normal factor in
$\det(\id-dc_t)$ is
$\det(2\id_{N_pF})>0$.  Since $\psi_t$ is the time-$t$ flow of $-\nabla h$, at a critical point
$p$ one has
\[
 d\psi_t|_p=e^{-t \operatorname{Hess}_p h}.
\]
  Hence, if
$\lambda_1,\ldots,\lambda_{\dim F}$ are the eigenvalues of $\operatorname{Hess}_p h$, then
\[
 \operatorname{sgn}\det(\id-d\psi_t|_p)
 =
 \prod_i\operatorname{sgn}(1-e^{-t\lambda_i})
 =
 \prod_i\operatorname{sgn}(\lambda_i)
 =
 \operatorname{sgn}\det(\operatorname{Hess}_p h)
 =
 (-1)^{\operatorname{ind}_h(p)}.
\]The Lefschetz number is invariant under homotopy, and therefore the total
contribution of $F$ is
\[
 \sum_{p\in\operatorname{Crit}(h)}
 (-1)^{\operatorname{ind}_h(p)}
 =
 \chi(F)
\]
by Morse theory.  Performing this perturbation independently near each fixed
component gives
\[
 L(c)=\sum_{F\subset S_{\mathbb R}}\chi(F)
 =\chi(S_{\mathbb R}).
\]}
\]
Consequently
\[
 \chi(S_{\mathbb R})=1.
\]
In particular, $S_{\mathbb R}$ has a component diffeomorphic to either $S^2$ or
$\mathbb{RP}^2$.

Choose a component
\[
 F\subset S_{\mathbb R}
\]
diffeomorphic to either $S^2$ or $\mathbb{RP}^2$, and choose a point
$x\in F$.  Let
\[
 p:\mathbb B_{\mathbb C}^2\longrightarrow S
\]
be the universal covering and choose a lift $\widetilde x\in p^{-1}(x)$.
There is a unique lift
\[
 \widetilde c:\mathbb B_{\mathbb C}^2\longrightarrow
 \mathbb B_{\mathbb C}^2
\]
of $c$ satisfying $\widetilde c(\widetilde x)=\widetilde x$.  Since
$\widetilde c^2$ is a deck transformation fixing $\widetilde x$, one has
$\widetilde c^2=\id$, so $\widetilde c$ is again an involution.

Let $\widetilde F$ be the connected component of
$\operatorname{Fix}(\widetilde c)$ containing $\widetilde x$.  The
restriction
\[
 p|_{\widetilde F}:\widetilde F\longrightarrow F
\]
is a connected covering.  Indeed, if $\gamma$ is a path in $F$ starting at
$x$ and $\widetilde\gamma$ is its lift starting at $\widetilde x$, then
$c\circ\gamma=\gamma$, while both $\widetilde\gamma$ and
$\widetilde c\circ\widetilde\gamma$ are lifts of $\gamma$ with the same
initial point.  Uniqueness of path lifting therefore gives
\[
 \widetilde c\circ\widetilde\gamma=\widetilde\gamma,
\]
so $\widetilde\gamma$ lies entirely in $\widetilde F$.  Hence
$p(\widetilde F)=F$.  Since $F$ is either $S^2$ or $\mathbb{RP}^2$, its
fundamental group is finite, and therefore every connected covering of $F$
has finitely many sheets.  Thus $\widetilde F$ is a finite cover of the
compact surface $F$ and is itself compact.
 This is impossible by
Smith theory.  Indeed, since
$\mathbb B_{\mathbb C}^2$
is mod-$2$ acyclic, Smith theory implies the fixed-point set of the lifted involution is also
mod-$2$ acyclic.  In particular, if it is nonempty then it is connected.
On the other hand, an anti-holomorphic involution has a smooth real
two-dimensional fixed locus.  If this fixed locus were compact, then it would
be a closed connected surface and hence 
\[
 H_2\bigl(\operatorname{Fix}(\widetilde c);\mathbb F_2\bigr)
 \cong \mathbb F_2,
\]
contradicting mod-$2$ acyclicity.  Thus the fixed-point set upstairs cannot
have a compact component.  This contradiction excludes anti-holomorphic
involutions.

To pass from involutions to arbitrary anti-holomorphic diffeomorphisms,
Kharlamov and Kulikov prove that
\[
 \operatorname{Aut}(S)
\]
contains no element of even order
\cite[Lemma~5.1]{KharlamovKulikovReal}.  We recall the argument, since it is
particularly short.  It is enough to exclude a non-trivial holomorphic involution
\[
 h:S\longrightarrow S.
\]
Since $h$ is a non-trivial holomorphic involution, every
nonempty connected component of $\operatorname{Fix}(h)$ is a smooth complex
submanifold of $S$ of
dimension either $0$ or $1$.
A one-dimensional component $C$ of $\operatorname{Fix}(h)$ would satisfy,
by the Enoki--Hirzebruch relative proportionality theorem,
\[
 \chi(C)=2C^2.
\]
Since $b_2(S)=1$ and $K^2=9\neq0$, there is some $r\in\mathbb Q$ such that
\[
 [C]=rK
 \qquad\text{in }H^2(S;\mathbb Q).
\]
Hence
\[
 C^2=r^2K^2\geq0.
\]
Since $C$ is a nonempty complex curve and $K_S$ is ample, the
Nakai--Moishezon criterion gives
\[
 K\cdot C>0.
\]
The adjunction formula therefore yields
\[
 2g(C)-2=C^2+K\cdot C>0,
\]
so
\[
 \chi(C)=2-2g(C)<0.
\]
This contradicts
\[
 \chi(C)=2C^2\geq0.
\]
Hence $\operatorname{Fix}(h)$ has no one-dimensional components.
The topological Lefschetz formula then shows that $h$ has exactly three
isolated fixed points.  Indeed, $h$ acts trivially on
$H^0(S;\mathbb Q)$ and $H^4(S;\mathbb Q)$ and, since it preserves the
canonical class, also on the one-dimensional group $H^2(S;\mathbb Q)$.
On the other hand, the holomorphic Lefschetz formula applied to
$\mathcal O_S$ has left-hand side
\[
 \sum_{i=0}^{2}(-1)^i
 \operatorname{tr}
 \bigl(h^*:H^i(S,\mathcal O_S)\to H^i(S,\mathcal O_S)\bigr)
 =1,
\]
because $p_g=q=0$.  At each isolated fixed point the differential of the
involution is $-\id$; otherwise a $+1$ eigendirection would give a
positive-dimensional fixed locus.  Thus each fixed point contributes
\[
 \frac{1}{\det(\id-(-\id))}=\frac14
\]
to the holomorphic Lefschetz formula.  The three fixed points would therefore
give
\[
 \frac34=1,
\]
a contradiction.  Hence $\operatorname{Aut}(S)$ has no element of even
order.

Finally, suppose that $g:S\to S$ were anti-holomorphic.  Since $K_S$ is
ample, $\operatorname{Aut}(S)$ is finite, and $g^2$ is a holomorphic
automorphism.  Thus $g$ has finite even order, say
\[
 \operatorname{ord}(g)=2m.
\]
Then $g^2$ has order $m$.  Since $\operatorname{Aut}(S)$ has no element of
even order, $m$ must be odd.  It follows that
\[
 g^m
\]
is an anti-holomorphic involution, contradicting what we have shown.
Therefore the anti-holomorphic alternative in the strong-rigidity theorem is
impossible.

It follows that the diffeomorphism $g$ homotopic to $G$ is holomorphic.
Hence it preserves the canonical line bundle:
\[
 g^*K_S\cong K_S.
\]
Thus, 
\[
 G^*K=g^*K=K
 \qquad\text{in }H^2(S;\mathbb Z).
\]
This in fact proves the canonical-class rigidity for every
self-homeomorphism of $S$; in the arguments below we only need it for the
orientation-preserving self-homeomorphisms obtained by gluing across the
divisor cap.
\end{proof}

\begin{lemma}\label{lem:bicanonical}
A general member $D\in|2K|$ of the complete linear system \footnote{The complete linear system $|2K|$ is
the projective space of nonzero sections of $K_S^{\otimes2}$ modulo
multiplication by nonzero scalars, with a section identified with its zero
divisor.  Its base locus is the set of points where all sections vanish; the
system is base-point free when this set is empty.}  $|2K|$ is a smooth connected
curve satisfying
\[
 D^2=36,
 \qquad
 g(D)=28.
\]
\end{lemma}

\begin{proof}
We first show that $|2K|$ is base-point free.  Apply the base-point part of Reider's theorem \cite[Theorem 1]{Reider} to the nef line
bundle\footnote{A holomorphic line bundle $L$ on a smooth projective variety $X$
is called \emph{nef} (numerically effective) if
\[
 c_1(L)\cdot C\geq0
\]
for every irreducible complex curve $C\subset X$.  Every ample line bundle is nef.} $L=K$.
Reider's theorem says that if $L$ is a nef line bundle on a smooth
projective surface with $L^2\geq5$ and the adjoint linear system
$|K_S+L|$ has a base point, then there exists a nonzero effective divisor
$E$ such that either
\[
 L\cdot E=0,\qquad E^2=-1,
\]
or
\[
 L\cdot E=1,\qquad E^2=0.
\]

The first alternative contradicts ampleness by the Nakai--Moishezon criterion.  In the second alternative,
adjunction for the arithmetic genus of the effective divisor $E$ gives
\[
 2p_a(E)-2=E^2+K\cdot E=1.
\]
The left-hand side is even because $p_a(E)\in\Z$, regardless of whether $E$ is
irreducible or reduced, so this is impossible.  Thus $|2K|$ is base-point free.

Bertini's theorem says that a general member of a linear system on a smooth
complex variety is smooth away from the base locus; see
\cite[Chapter~III, Section~10, especially Corollary~10.9]{Hartshorne}.  Since
the base locus of $|2K|$ is empty, a general member $D$ is smooth. 

Next, let us show it is
connected.  The divisor sequence
\[
 0\longrightarrow\mathcal{O}_S(-2K)
 \longrightarrow\mathcal{O}_S
 \longrightarrow\mathcal{O}_D
 \longrightarrow0
\]
gives the cohomology exact sequence
\[
 0\longrightarrow H^0(S,\mathcal O_S(-2K))
 \longrightarrow H^0(S,\mathcal O_S)
 \longrightarrow H^0(D,\mathcal O_D)
 \longrightarrow H^1(S,\mathcal O_S(-2K)).
\]
Ampleness gives $H^0(S,\mathcal O_S(-2K))=0$.  By Serre duality and Kodaira
vanishing,
\[
 H^1(S,\mathcal O_S(-2K))^*
 \cong H^1(S,\mathcal O_S(3K))=0.
\]
Hence restriction induces
$H^0(S,\mathcal O_S)\cong H^0(D,\mathcal O_D)$, so
$H^0(D,\mathcal O_D)=\C$.  A smooth projective curve has one independent
constant function on each connected component; therefore $D$ is connected.
Finally,
\[
 D^2=(2K)^2=4K^2=36,
\]
and adjunction gives
\[
 2g(D)-2=D\cdot(D+K)=2K\cdot3K=6K^2=54.
\]
Thus $g(D)=28$.
\end{proof}

Fix a fake projective plane $S$ and a smooth connected member $D\in|2K|$.  By the definition of the complete
linear system $|2K|$, there is a nonzero section
$s\in H^0(S,K_S^{\otimes2})$, unique up to multiplication by a nonzero scalar,
whose zero divisor is $D$. 
Since $D=\operatorname{div}(s)$ is a smooth divisor, $s$ is
transverse to the zero section along $D$.
Fix such an $s$ for the remainder of the paper.

\section{The conjugate Stein domains}\label{sec:stein}

Let $s$ be the section fixed at the end of
Section~\ref{sec:algebraic}.  Choose a Hermitian metric $h$ on the ample line
bundle $2K$ with positive Chern curvature.  On $S\setminus D$, define
\begin{equation}\label{eq:Phi}
 \Phi=-\log\norm{s}_h^2.
\end{equation}
Since the zero divisor of $s$ is exactly $D$, one has
\[
 \lVert s(x)\rVert_h^2\longrightarrow0
 \qquad\text{as }x\to D.
\]
Hence
\[
 \Phi(x)=-\log\lVert s(x)\rVert_h^2\longrightarrow+\infty
 \qquad\text{as }x\to D.
\]
Therefore, for every $c\in\mathbb R$, the sublevel set
\[
 \{\Phi\leq c\}
\]
is disjoint from some neighborhood of $D$.  Since it is closed in
$S\setminus D$, it is consequently closed in $S$.  As $S$ is compact,
$\{\Phi\leq c\}$ is compact. 
Thus $\Phi:S\setminus D\to\mathbb R$ is
proper.

We also verify its Levi form.  
Since $s$ is nowhere zero on $S\setminus D$, it is a global holomorphic
frame of $(2K)|_{S\setminus D}$.  
Since $\norm{s}_h^2=e^{-\Phi}$, we have
\[
 \Omega_h
 =\sqrt{-1}\,\partial\bar\partial\Phi
 =\frac12dd_J^c\Phi.
\]
Thus $dd_J^c\Phi$ is twice the positive Chern curvature of $(2K,h)$, and
$\Phi$ is a proper strictly $J$-plurisubharmonic function on
$S\setminus D$.

Choose a sufficiently large regular value $c$ and set
\begin{equation}\label{eq:Xdef}
 X=\{\Phi\leq c\},
 \qquad
 \rho_0=\Phi-c.
\end{equation}
Then $X$ is diffeomorphic to $S\setminus\operatorname{int}\nu(D)$ for a
closed tubular neighborhood $\nu(D)$, and $\rho_0\leq0$ is a strictly
plurisubharmonic defining function.  Consequently $(X,J|_X,\rho_0)$ is a
compact Stein datum in the sense of Section~\ref{sec:conventions} and $S=X\cup \nu(D)$.

Define
\[
 J_+=J|_X,
 \qquad
 J_-=-J|_X,
\]
and initially put
\begin{equation}\label{eq:conjugateforms}
 \lambda_+^0=d_{J_+}^c\rho_0,
 \qquad
 \lambda_-^0=d_{J_-}^c\rho_0=-\lambda_+^0.
\end{equation}
Let $Y=\partial X$ and define the initial boundary contact forms and their
cooriented kernels by
\begin{equation}\label{eq:initial-boundary-data}
 \alpha_\pm^0=\lambda_\pm^0|_Y,
 \qquad
 \xi_\pm=\ker\alpha_\pm^0.
\end{equation} 
The sign reversal causes no conflict: both the complex structure and the
compatible symplectic form are reversed.  Since the complex dimension is two, the
 orientations of $X$ induced by $J_+$ and $J_-$ are the same, and
\[
 d\lambda_-^0(v,J_-v)
 =(-d\lambda_+^0)(v,-J_+v)
 =d\lambda_+^0(v,J_+v)>0.
\]
 Thus $(X,J_-,\rho_0)$ is also a Stein datum.

In the normalization process of Section~\ref{sec:normalization}, we deform the common defining function
$\rho_0$ in two different ways, through strictly
$J_\pm$-plurisubharmonic defining functions respectively, so that the resulting boundary
forms are pullbacks of one fixed Boothby--Wang connection form. The complex
structures $J_\pm$ remain fixed, and on the boundary each form is only
multiplied by a positive function.  Hence the cooriented contact plane fields, 
the first Chern classes, and the canonical $\SpinC$ structures do not change
during normalization.  We therefore compute the latter two invariants now,
using the common initial defining function $\rho_0$.
\begin{proposition}\label{prop:c1equal}
The first Chern classes agree:
\[
 c_1(TX,J_+)=c_1(TX,J_-)
 \quad\text{in }H^2(X;\Z).
\]
\end{proposition}

\begin{proof}
The section $s$ is nowhere zero on $X$ and trivializes $(2K)|_X$.  Therefore
\begin{equation}\label{eq:twotorsion}
 2K|_X=0\quad\text{in }H^2(X;\Z).
\end{equation}
Complex conjugation changes the sign of the first Chern class, so
\[
 c_1(TX,J_+)=-K|_X,
 \qquad
 c_1(TX,J_-)=K|_X.
\]
Equation \eqref{eq:twotorsion} gives $K|_X=-K|_X$.
\end{proof}

\begin{lemma}[Conjugation of a $\SpinC$ structure]
\label{lem:spinc-conjugation}
For every $\SpinC$ structure $\mathfrak s$ on an oriented four-manifold $X$,
conjugation satisfies 
\[
 \overline{\mathfrak s}=\mathfrak s-c_1(\mathfrak s)
 \quad\text{in the }H^2(X;\Z)\text{-torsor of }\SpinC\text{ structures}.
\]
\end{lemma}

\begin{proof}
Choose a good cover and write the transition functions of the principal
$\SpinC(4)$ bundle of $\mathfrak s$ as
\[
 [g_{ij},z_{ij}]\in
 (\operatorname{Spin}(4)\times U(1))/\{(1,1),(-1,-1)\}.
\]
The determinant homomorphism sends $[g,z]$ to $z^2$, so the determinant line
$L=\det(\mathfrak s)$ has transition functions $z_{ij}^2$. On the other hand, $\overline{\mathfrak s}$ has transitions
$[g_{ij},z_{ij}^{-1}]$.

Twisting $\mathfrak s$ by $L^{-1}$ multiplies the central component by the
transition functions $z_{ij}^{-2}$ of $L^{-1}$.  Its transitions are therefore
\[
 [g_{ij},z_{ij}z_{ij}^{-2}]
 =[g_{ij},z_{ij}^{-1}].
\]
Thus
\[
 \overline{\mathfrak s}=\mathfrak s\otimes L^{-1}.
\]
Twisting by a line bundle of first Chern class $a$ is translation by $a$ in the
affine $H^2(X;\Z)$-torsor of $\SpinC$ structures, and
$c_1(L)=c_1(\mathfrak s)$.  This proves the stated equality exactly, not merely
after quotienting by $2$-torsion.
\end{proof}

\begin{proposition}[The canonical $\SpinC$ structures]\label{prop:spincdifferent}
Let $k_X=K|_X$.  Then $k_X$ is a nonzero element of order two and
\begin{equation}\label{eq:spincdifference}
 \mathfrak{s}_{J_-}
 =\overline{\mathfrak{s}_{J_+}}
 =\mathfrak{s}_{J_+}+k_X
 \neq\mathfrak{s}_{J_+}.
\end{equation}
\end{proposition}

\begin{proof}
Identify $X$ with
$S\setminus\Int\nu(D)$, where $\nu(D)$ is the closed tubular neighborhood
removed from $S$.  Let $i:D\hookrightarrow S$ and $j:X\hookrightarrow S$
denote the inclusions.  Excision gives
\begin{equation}\label{eq:excision-pair}
 H^2(S,X;\Z)\cong H^2(\nu(D),\partial\nu(D);\Z),
\end{equation}
while the Thom isomorphism for the real rank-two normal bundle, oriented by
its complex structure, gives
\begin{equation}\label{eq:Thom-isomorphism}
 H^0(D;\Z)\xrightarrow{\cong}
 H^2(\nu(D),\partial\nu(D);\Z).
\end{equation}
Under the composite identification, the long exact sequence of the pair
$(S,X)$ contains
\begin{equation}\label{eq:Thomrestriction}
 H^0(D;\Z)\xrightarrow{i_!}H^2(S;\Z)
 \xrightarrow{j^*}H^2(X;\Z),
 \qquad
 i_!(1)=\PD_S[D]=2K.
\end{equation}
If $k_X=K|_X$ is zero, exactness would imply $K=2mK$ for some $m\in\Z$.  This is
impossible because $K$ is non-torsion, as $K^2=9$.  Thus $k_X\neq0$; together
with \eqref{eq:twotorsion}, this shows that $k_X$ has order exactly two.

The relation $J_-=-J_+$ implies
$\mathfrak s_{J_-}=\overline{\mathfrak s_{J_+}}$.
Since $c_1(\mathfrak{s}_{J_+})=-k_X$, Lemma~\ref{lem:spinc-conjugation}
gives
\[
 \mathfrak{s}_{J_-}
 =\overline{\mathfrak{s}_{J_+}}
 =\mathfrak{s}_{J_+}+k_X.
\]
\end{proof}

\section{The boundary as a negative prequantization space}\label{sec:boundary}

Let us examine the contact structures $\xi_\pm$ on the circle bundle $Y=\partial X=-\partial \nu (D)$.

Let  $\pi:P\to\Sigma$ be a smooth principal $S^1$ bundle over a manifold.  A connection form is an $S^1$-invariant
1-form $\eta$ on $P$ satisfying $\eta(T)=1$, where $T$ is the positively oriented
infinitesimal generator of the circle action.  The invariance of $\eta$ gives
$\mathcal L_T\eta=0$.  Hence Cartan's formula implies
\[
 0=\mathcal L_T\eta
 =\iota_Td\eta+d(\iota_T\eta)
 =\iota_Td\eta+d(1)=\iota_Td\eta.
\]
Moreover,
\[
 \mathcal L_T(d\eta)=d(\mathcal L_T\eta)=0.
\]
Thus $d\eta$ is both horizontal and $S^1$-invariant, hence basic.  There is
therefore a unique 2-form $\Omega$ on $\Sigma$ such that
\[
 d\eta=\pi^*\Omega.
\]
We identify basic forms on $P$ with their corresponding forms on $\Sigma$.
Our Euler-class convention is
\begin{equation}\label{eq:eulerconvention}
 e(P\to\Sigma)
 =-\left[\frac{d\eta}{2\pi}\right]
 =-\left[\frac{\Omega}{2\pi}\right]
 \in H^2(\Sigma;\Z).
\end{equation}
Thus a connection form with positive curvature over an oriented surface has
negative Euler class.

Suppose now that $\Sigma$ is a closed oriented surface and that $\Omega$ is a
positive area form.  If $(v_1,v_2)$ is a positively oriented horizontal basis,
then
\[
 (\eta\wedge d\eta)(T,v_1,v_2)
 =\Omega(d\pi(v_1),d\pi(v_2))>0.
\]
Hence $\eta$ is a contact form.  Since $\eta(T)=1$ and
$\iota_Td\eta=0$, its Reeb vector field is $T$.  We call its kernel the
Boothby--Wang, or positive-curvature prequantization, contact structure \cite{BoothbyWang}.
 Reeb orbits of $\eta$ are the positively oriented circle fibers.
We call $(P,\ker\eta,t)$ a prequantization triple, where $t\in [S^1, P]$ denotes the
oriented free-homotopy class of a positively oriented circle fiber.

A holomorphic line bundle is called \emph{positive} if it admits a Hermitian
metric $h$ whose Chern curvature form $\Omega_h$ satisfies
\[
 \Omega_h(v,Jv)>0
\]
for every nonzero tangent vector $v$.
\begin{proposition}[Boundary prequantization model]\label{prop:boundarymodel}
Let $(S,J)$ be a compact complex surface, let $L\to S$ be a positive
holomorphic line bundle, and let $s\in H^0(S,L)$ have zero set
\[
 D=s^{-1}(0),
\]
where $D$ is a smooth connected curve along which $s$ vanishes transversely.
Put
\[
 \Phi=-\log\norm{s}_h^2
\]
for a positive Hermitian metric $h$, and assume $D^2=n>0$.  Then:
\begin{enumerate}[label=\textup{(\roman*)}]
\item
Every sufficiently large $c$ is a regular value of $\Phi$, and hence
\[
 Y_c=\{\Phi=c\}
\]
is a smooth hypersurface.

\item
Fix one $c$ as in (i).  For $J_+=J$ and $J_-=-J$, define
\[
 \alpha_{\pm,c}=d_{J_\pm}^c\Phi|_{Y_c}.
\]
Then $(Y_c,\ker\alpha_{+,c})$ is contactomorphic to the Boothby--Wang
contact structure \cite{BoothbyWang} on the oriented circle bundle of Euler
number $-n$ over
\[
 D_+=(D,J|_D).
\]

\item
If $u_+$ is the positive complex-normal circle\footnote{Precisely,
$u_+$ is the oriented free-homotopy class of a fiber of the unit circle bundle
$\mathbb S(N_{D/S})\to D$, oriented by multiplication by
$e^{\sqrt{-1}\theta}$ in the complex normal line.  Equivalently, it is the
fiberwise boundary orientation of the unit disk in that complex line.} and $t_+$ is the positive Boothby--Wang fiber, then
\[
 t_+=u_+^{-1}.
\]

\item
$(Y_c,\ker\alpha_{-,c})$ is contactomorphic to the corresponding
Boothby--Wang structure over
\[
 D_-=(D,-J|_D),
\]
again with Euler number $-n$.

\item
On the same underlying boundary,
\[
 [t_-]=[t_+^{-1}],
\]
where $t_-$ is the positive Boothby--Wang fiber of (iv).
\item
There is a coorientation-preserving contactomorphism
\[
 \tau_0:(Y_c,\ker\alpha_{+,c})
 \longrightarrow
 (Y_c,\ker\alpha_{-,c})
\]
satisfying
\[
 (\tau_0)_*[t_+]=[t_-].
\]
\end{enumerate}
\end{proposition}

\begin{proof}
Let $N=N_{D/S}$ be the complex normal line bundle, let
$\pi_L:L\to S$ be the bundle projection, and let $z:S\to L$ be the zero
section.  Along $D$ one has $s=z$.  Hence
\[
 (ds-dz)|_D:TS|_D\longrightarrow TL|_{z(D)}
\]
takes values in the vertical tangent bundle
\[
 \ker(d\pi_L)|_{z(D)}\cong L|_D
\]
and vanishes on $TD$.  It therefore descends to a canonical complex bundle
map
\[
 \delta=d^\perp s:N\longrightarrow L|_D.
\]
Transversality of $s$ says exactly that $\delta$ is an isomorphism.  Since
$s$ is holomorphic, $\delta$ is a holomorphic line-bundle isomorphism.

\smallskip
\noindent\emph{Proof of \textup{(i)}.}
Choose a Hermitian metric on $N$, a tubular exponential map
\[
 \exp:N_\varepsilon\longrightarrow S,
\]
and a Hermitian connection on $L$.  For
$x\in D$, $v\in\mathbb S(N_x)$, and $r\geq0$ sufficiently small, let
\[
 P_{x,v,r}:L_{\exp_x(rv)}\longrightarrow L_x
\]
denote parallel transport along the radial path
\[
 t\longmapsto\exp_x(tv),\qquad 0\leq t\leq r.
\]
By Taylor's formula with integral remainder,
\begin{equation}\label{eq:normal-Taylor}
 P_{x,v,r}\bigl(s(\exp_x(rv))\bigr)
 =
 r\delta_x(v)+r^2R(x,v,r),
\end{equation}
where $R$ depends smoothly on $(x,v,r)$.  Since the connection is Hermitian,
$P_{x,v,r}$ is unitary.  Hence, setting
\[
 A(x,v,r)
 =
 \bigl\|\delta_x(v)+rR(x,v,r)\bigr\|_{h_x}^2,
\]
we obtain
\[
 \norm{s(\exp_x(rv))}_h^2
 =
 r^2A(x,v,r).
\]
The function $A$ is smooth and satisfies
\[
 A(x,v,0)=\norm{\delta_x(v)}_{h_x}^2>0.
\]
Since $\mathbb S(N)$ is compact and $\delta$ is an isomorphism, there is
$m>0$ such that
\[
 A(x,v,0)\geq m
\]
for every $x\in D$ and $v\in\mathbb S(N_x)$.
By uniform continuity of $A$ on
$\mathbb S(N)\times[0,\varepsilon_0]_r$, there exists
$0<\varepsilon\leq\varepsilon_0$ such that
\[
 A(x,v,r)\geq\frac{m}{2}
\]
for all $x\in D$, $v\in\mathbb S(N_x)$, and
$0\leq r<\varepsilon$.
Moreover, compactness implies that $\partial_rA$ is uniformly bounded on
this set.  Decreasing $\varepsilon$ once more if necessary, we therefore have
\[
 2A(x,v,r)+r\,\partial_rA(x,v,r)>0
\]
for all $x\in D$, $v\in\mathbb S(N_x)$, and $0\leq r<\varepsilon$.
Consequently,
\[
 \frac{\partial}{\partial r}
 \norm{s(\exp_x(rv))}_h^2
 =
 r\bigl(2A(x,v,r)+r\,\partial_rA(x,v,r)\bigr)>0
\]
whenever
\[
 x\in D,\qquad
 v\in\mathbb S(N_x),\qquad
 0<r<\varepsilon.
\]
Thus
\[
 d\norm{s}_h^2\neq0
\]
throughout the punctured tubular neighborhood
\[
 U\setminus D
 =
 \{\exp_x(rv)\mid x\in D,\ v\in\mathbb S(N_x),\ 0<r<\varepsilon\}.
\]

Since $s^{-1}(0)=D$ and $S\setminus U$ is compact, there exists
$\mu>0$ such that
\[
 \norm{s}_h^2\geq\mu
 \qquad\text{on }S\setminus U.
\]
Choose $c_0$ sufficiently large that
\[
 e^{-c_0}<\mu.
\]
Then, for every $c\geq c_0$,
\[
 Y_c
 =
 \{\Phi=c\}
 =
 \{\norm{s}_h^2=e^{-c}\}
 \subset U\setminus D.
\]
On $Y_c$ one has
\[
 d\Phi
 =
 -\frac{d\norm{s}_h^2}{\norm{s}_h^2}
 \neq0.
\]
Thus every $c\geq c_0$ is a regular value of $\Phi$. This proves \textup{(i)}.

\smallskip
We now prepare the prequantization model.  Pull the metric $h|_D$ back by
$\delta$ to a Hermitian metric on $N$; then $\delta$ is unitary.

Let $\mathbb S(L)\subset L$ be the unit circle bundle, with projection still
denoted by $\pi_L$.  
On $L^\times=L\setminus z(S)$, let
\[
 v\in\Gamma(L^\times,\pi_L^*L)
\]
be the tautological section, defined by
\[
 v_\ell=\ell\in L_{\pi_L(\ell)}
       =(\pi_L^*L)_\ell
 \qquad (\ell\in L^\times).
\]
Thus
\[
 \norm{v}_h^2(\ell)=\norm{\ell}_{h_{\pi_L(\ell)}}^2.
\]
Put
\[
 \widetilde\vartheta_L
 =
 \frac12d_{J_L}^c\log\norm{v}_h^2,
 \qquad
 \vartheta_L
 =
 \widetilde\vartheta_L|_{\mathbb S(L)},
\]
where $J_L$ is the complex structure on the total space of $L$.  The form
$\widetilde\vartheta_L$ is invariant under multiplication by positive real
numbers, and $\vartheta_L$ is the real Chern connection form. 
If $U_L$ is the generator of fiberwise multiplication by
$e^{\sqrt{-1}\theta}$, then
\[
 \vartheta_L(U_L)=1,
 \qquad
 d\vartheta_L=-\pi_L^*\Omega_L,
 \qquad
 \left[\frac{\Omega_L}{2\pi}\right]=c_1(L),
\]
where $\Omega_L$ is the positive Chern curvature form.  Indeed, choose a local
holomorphic frame $e$ of $L$ over an open set $U\subset S$ and write
\[
 \norm{e}_h^2=e^{-\psi}.
\]
Using the corresponding fiber coordinate $w$ on $\pi_L^{-1}(U)$, the
tautological section is
\[
 v=w e,
\]
and hence
\[
 \norm{v}_h^2=|w|^2e^{-\pi_L^*\psi}.
\]
Writing $w=re^{\sqrt{-1}\theta}$ and using
\[
 \frac12d^c\log|w|^2=d\theta,
\]
we obtain
\[
 \widetilde\vartheta_L
 =
 \frac12d^c\log\norm{v}_h^2
 =
 d\theta-\frac12\pi_L^*d^c\psi.
\]
Since $U_L=\partial_\theta$, this gives
\[
 \widetilde\vartheta_L(U_L)=1,
\]
and therefore
\[
 \vartheta_L(U_L)=1.
\]
Moreover,
\[
 d\widetilde\vartheta_L
 =
 -\frac12\pi_L^*dd^c\psi
 =
 -\pi_L^*\Omega_L,
\]
where
\[
 \Omega_L
 =
 \sqrt{-1}\,\partial\bar\partial\psi
 =
 \frac12dd^c\psi.
\]
Restricting to $\mathbb S(L)$ yields
\[
 d\vartheta_L=-\pi_L^*\Omega_L.
\]
Finally, the Chern--Weil formula gives
\[
 \left[\frac{\Omega_L}{2\pi}\right]=c_1(L).
\]

Write
\[
 \pi:P=\mathbb S(N)\longrightarrow D
\]
for the unit normal circle bundle.  Transporting the connection by $\delta$
gives a form $\vartheta_N$ on $P$ satisfying
\[
 \vartheta_N(U_+)=1,
 \qquad
 d\vartheta_N=-\pi^*\Omega_N,
 \qquad
 \left[\frac{\Omega_N}{2\pi}\right]=c_1(N),
\]
where $\Omega_N$ is positive on $D_+$.

On $S\setminus D$, consider the normalized section
\[
 \widehat s
 =
 \frac{s}{\norm{s}_h}
 :
 S\setminus D\longrightarrow\mathbb S(L).
\]
Let
\[
 q:L^\times\longrightarrow\mathbb S(L),
 \qquad
 q(\ell)=\frac{\ell}{\norm{\ell}_h},
\]
be the fiberwise radial projection.  The local formula
\[
 \widetilde\vartheta_L
 =
 d\theta-\frac12\pi_L^*d^c\psi
\]
shows that $\widetilde\vartheta_L$ is independent of the radial coordinate
and has no radial component.  Hence
\[
 \widetilde\vartheta_L=q^*\vartheta_L.
\]
On the other hand, regarding the holomorphic section $s$ as a map
$S\setminus D\to L^\times$, one has
\[
 s^*\norm{v}_h^2=\norm{s}_h^2.
\]
Naturality of $d^c$ under the holomorphic map $s$ therefore gives
\[
 s^*\widetilde\vartheta_L
 =
 \frac12d_J^c\log\norm{s}_h^2
 =
 -\frac12d_J^c\Phi.
\]
Since
\[
 \widehat s=q\circ s,
\]
we conclude that
\begin{equation}\label{eq:intrinsic-angular-form}
 d_J^c\Phi
 =
 -2s^*\widetilde\vartheta_L
 =
 -2\widehat s^{\,*}\vartheta_L.
\end{equation}
\smallskip
\noindent\emph{Proof of \textup{(ii)} and \textup{(iii)}.}
Let
\[
 \widetilde S=[S;D]
\]
be the real oriented blow-up of $S$ along $D$, obtained by replacing each
point of $D$ by the circle of oriented rays in its real normal plane.  The
chosen Hermitian metric on $N$ identifies
\[
 \partial\widetilde S=\mathbb S(N)=P.
\]

The Taylor expansion \eqref{eq:normal-Taylor} shows that
$\norm{s}_h$ lifts to a smooth boundary defining function on
$\widetilde S$.  Indeed, in the coordinates $(x,v,r)$ induced by the tubular
neighborhood,
\[
 s(\exp_x(rv))
 =
 r\bigl(\delta_x(v)+rR(x,v,r)\bigr),
\]
so
\[
 \norm{s(\exp_x(rv))}_h
 =
 r\,a(x,v,r),
\]
where
\[
 a(x,v,r)
 =
 \norm{\delta_x(v)+rR(x,v,r)}_h
\]
is smooth and positive for $r$ sufficiently small.  The same formula gives
\[
 \frac{s(\exp_x(rv))}{\norm{s(\exp_x(rv))}_h}
 =
 \frac{\delta_x(v)+rR(x,v,r)}
 {\norm{\delta_x(v)+rR(x,v,r)}_h},
\]
so $\widehat s$ extends smoothly to $\partial\widetilde S$.  Its boundary
value is the unit circle bundle map
\[
 \mathbb S(\delta):
 \mathbb S(N)\longrightarrow\mathbb S(L|_D)
\]
induced by $\delta$.

It follows from \eqref{eq:intrinsic-angular-form} that $d_J^c\Phi$ extends
smoothly to $\widetilde S$, with boundary value
\[
 -2\vartheta_N.
\]
Put
\[
 T_+=-U_+,
 \qquad
 \eta_+=-\vartheta_N.
\]
Then
\[
 \eta_+(T_+)=1,
 \qquad
 d\eta_+=\pi^*\Omega_N,
\]
so $\eta_+$ is a positive-curvature Boothby--Wang contact form over $D_+$.

Since the lift of $\norm{s}_h$ is a boundary defining function, a collar of
$\partial\widetilde S$ identifies its sufficiently small positive level sets
with
\[
 Y_c
 =
 \{\norm{s}_h=e^{-c/2}\}.
\]
For $c$ sufficiently large, denote the resulting diffeomorphism by
\[
 \iota_c:P\longrightarrow Y_c.
\]
Smooth extension to the blow-up gives
\[
 \iota_c^*\alpha_{+,c}
 \longrightarrow
 2\eta_+
 \qquad\text{in }C^\infty
 \qquad(c\to\infty).
\]
The contact condition is open.  Hence, after increasing the lower bound on
$c$ if necessary, the straight-line path
\[
 (1-t)\,2\eta_+
 +
 t\,\iota_c^*\alpha_{+,c},
 \qquad
 0\leq t\leq1,
\]
consists of contact forms.  Gray stability \cite{Gray} therefore gives a
contactomorphism
\[
 (Y_c,\ker\alpha_{+,c})
 \cong
 (P,\ker\eta_+)
\]
isotopic to the collar identification.

By the Euler-class convention \eqref{eq:eulerconvention},
\[
 e(P\to D_+)
 =
 -\left[\frac{d\eta_+}{2\pi}\right]
 =
 -c_1(N),
\]
and therefore
\[
 \langle e(P\to D_+),[D_+]\rangle
 =
 -\langle c_1(N),[D_+]\rangle
 =
 -D^2
 =
 -n.
\]
This proves \textup{(ii)}.  Since
\[
 T_+=-U_+,
\]
the corresponding oriented fiber classes satisfy
\[
 t_+=u_+^{-1},
\]
proving \textup{(iii)}.

\smallskip
\noindent\emph{Proof of \textup{(iv)} and \textup{(v)}.}
For $J_-=-J$ one has
\[
 d_{J_-}^c=-d_{J_+}^c,
 \qquad
 \alpha_{-,c}=-\alpha_{+,c}.
\]
Both the base complex orientation and the complex-normal orientation reverse.
Thus, on the same underlying real circle bundle,
\[
 U_-=-U_+,
 \qquad
 T_-=-U_-=U_+=-T_+,
 \qquad
 \eta_-=-\eta_+.
\]
The form
\[
 d\eta_-=-d\eta_+
\]
is positive with respect to the reversed orientation of $D_-$.  Moreover,
\[
 \left\langle
 -\left[\frac{d\eta_-}{2\pi}\right],
 [D_-]
 \right\rangle
 =
 -n.
\]
The preceding blow-up and Gray-stability argument, now applied to $J_-$,
proves \textup{(iv)}, while $T_-=-T_+$ gives
\[
 [t_-]=[t_+^{-1}],
\]
proving \textup{(v)}.

For clarity, all orientation relations used later are collected here.  The
symbols $u_\pm$ and $t_\pm$ denote oriented free-homotopy classes, while
$U_\pm$ and $T_\pm$ denote their infinitesimal generators.
\begin{equation}\label{eq:orientationtable}
\begin{array}{c|c|c}
 & J_+=J & J_-=-J \\
\hline
\text{base complex orientation}
 & [D_+] & [D_-]=-[D_+] \\
\text{complex-normal generator}
 & U_+ & U_-=-U_+ \\
\text{positive contact generator}
 & T_+=-U_+ & T_-=-U_-=U_+=-T_+ \\
\text{fiber classes}
 & u_+=t_+^{-1} & u_-=t_-^{-1} \\
\text{Euler number in the }t_\pm\text{-orientation}
 & -n & -n
\end{array}
\end{equation}
The boundary-fiber orientation of the normal disk bundle is $u_\pm$; the
positive contact fiber $t_\pm$ is its opposite.  
 Reversing the
base and fiber orientations simultaneously leaves the evaluated Euler number
unchanged.

\smallskip
\noindent\emph{Proof of \textup{(vi)}.}
Choose an orientation-preserving diffeomorphism
\[
 D_+\longrightarrow D_-.
\]
By \textup{(ii)} and \textup{(iv)}, the two $t_\pm$-oriented circle bundles
have the same Euler number, so the base map lifts to a
positive-fiber-preserving bundle isomorphism.  After identifying the two circle bundles by this bundle isomorphism, regard
$\eta_+$ and $\eta_-$ as connection forms on a single oriented circle bundle
\[
 \pi:P\longrightarrow D
\]
with common positive vertical generator $T$.  We have
\[
 \eta_+(T)=\eta_-(T)=1,
 \qquad
 d\eta_+=\pi^*\Omega_+,
 \qquad
 d\eta_-=\pi^*\Omega_-,
\]
where $\Omega_+$ and $\Omega_-$ are positive area forms on the oriented
surface $D$.  For $0\leq t\leq1$, set
\[
 \eta_t=(1-t)\eta_++t\eta_-.
\]
Then
\[
 \eta_t(T)=1
\]
and
\[
 d\eta_t
 =
 \pi^*\Omega_t,
 \qquad
 \Omega_t=(1-t)\Omega_++t\Omega_-.
\]
Since the convex combination of two positive area forms on an oriented
surface is again positive, $\Omega_t$ is positive for every $t$.  Hence
\[
 \eta_t\wedge d\eta_t>0,
\]
so $\{\eta_t\}_{0\leq t\leq1}$ is a path of contact forms.
  Gray
stability, composed with the contactomorphisms obtained in (ii) and (iv), gives a
coorientation-preserving contactomorphism $\tau_0$ carrying $[t_+]$ to
$[t_-]$.  
\end{proof}
Applying Proposition~\ref{prop:boundarymodel} to $L=2K$ gives:

\begin{corollary}\label{cor:boundarynumbers}
The contact boundaries of the two Stein data of
Section~\ref{sec:stein} are contactomorphic to the Boothby--Wang
prequantization contact structure $\xi_{\mathrm{BW}}$ on the oriented circle
bundle $Y_{28,-36}$ over a surface of genus $28$ with Euler number $-36$.

\end{corollary}
Let $t$ denote the oriented positive fiber class of $(Y_{28,-36}, \xi_{\mathrm{BW}})$.
By Proposition~\ref{prop:boundarymodel}, choose coorientation-preserving
contactomorphisms
\[
 \kappa_\pm:(Y,\xi_\pm,t_\pm)
 \longrightarrow
 (Y_{28,-36},\xi_{\mathrm{BW}},t)
\]
preserving the oriented positive fiber class.

\section{Strict normalization of the boundary contact forms}\label{sec:normalization}

Recall the initial data from Section~\ref{sec:stein}:
\[
 \lambda_\pm^0=d_{J_\pm}^c\rho_0,
 \qquad
 \alpha_\pm^0=\lambda_\pm^0|_Y,
 \qquad
 \xi_\pm=\ker\alpha_\pm^0.
\]
  Fix a positive-curvature connection form
$\eta$ for $\xi_{\mathrm{BW}}$ whose positively oriented Reeb orbits are the
fibers $t$.  We shall choose the final boundary forms to be
$\kappa_\pm^*\eta$.  This simultaneously gives strict normalization and
makes the oriented free-homotopy classes $[t_\pm]$ the classes of the
characteristic circles \footnote{Let $(X^{2n},\omega)$ be a symplectic manifold with boundary
$Y=\partial X$.  At each point $p\in Y$, the tangent space $T_pY$ is a
hyperplane in the symplectic vector space $(T_pX,\omega_p)$.  Its symplectic
orthogonal
\[
 (T_pY)^{\omega}
 =
 \{v\in T_pX\mid \omega_p(v,w)=0
 \text{ for every }w\in T_pY\}
\]
is one-dimensional, since
\[
 \dim (T_pY)^{\omega}
 =
 \dim T_pX-\dim T_pY
 =
 1.
\]
Moreover, every one-dimensional subspace of a symplectic vector space is
isotropic, so
\[
 (T_pY)^{\omega}
 \subset
 \bigl((T_pY)^{\omega}\bigr)^{\omega}
 =
 T_pY.
\]
It follows that
\[
 \ker(\omega_p|_{T_pY})
 =
 T_pY\cap(T_pY)^{\omega}
 =
 (T_pY)^{\omega},
\]
and hence $\omega|_{TY}$ has a one-dimensional kernel.  These kernels form
the \emph{characteristic line field} of $Y$.  Its integral curves are called
the \emph{characteristic curves}; a closed integral curve is called a
\emph{characteristic circle}.  If $\omega=d\lambda$ and
$\alpha=\lambda|_Y$ is a contact form, then
\[
 \ker(d\alpha)=\ker(\omega|_{TY}),
\]
so the characteristic circles are precisely the closed Reeb orbits of
$\alpha$, with their natural orientation.} of the final exact symplectic forms.

We will construct final defining functions $\rho_\pm$ and set
\[
 \lambda_\pm=d_{J_\pm}^c\rho_\pm,
 \qquad
 \omega_\pm=d\lambda_\pm,
 \qquad
 \alpha_\pm=\lambda_\pm|_Y.
\]
Only the defining functions are changed.  The complex structures $J_\pm$, the
cooriented contact hyperplane fields $\xi_\pm$, the first Chern classes, and the
canonical $\SpinC$ structures therefore remain those already fixed in
Section~\ref{sec:stein}.

\begin{lemma}[Rescaling a strictly pseudoconvex defining function]
\label{lem:definingrescale}
Let $(X,J)$ be a compact Stein domain with smooth boundary $Y$, and let
$\rho:X\to\R$ be a global strictly $J$-plurisubharmonic defining function
with $\rho\leq0$:
\[
 Y=\{\rho=0\},
 \qquad
 d\rho|_Y\neq0.
\]
Put $\alpha=d_J^c\rho|_Y$.  For every smooth function
$q:Y\to\R_{>0}$, there is a global strictly $J$-plurisubharmonic defining
function $\widetilde\rho$ with the same boundary such that
\[
 d_J^c\widetilde\rho|_Y=q\alpha.
\]
The two defining functions may be joined through strictly plurisubharmonic
defining functions whose boundary zero set remains $Y$ and whose boundary
differentials never vanish.
\end{lemma}
\begin{proof}
Extend
\[
 H=\log q
\]
smoothly to a collar of $Y$.  On this collar set
\begin{equation}\label{eq:rescaledrho}
 \rho_1=e^H\rho+M\rho^2,
\end{equation}
where $M>0$ will be chosen below.

\smallskip
\noindent\emph{Step 1: the boundary $1$-jet.}
The Leibniz rule gives
\[
 d_J^c\rho_1
 =
 e^H d_J^c\rho
 +\rho e^H d_J^cH
 +2M\rho\,d_J^c\rho.
\]
On the boundary
\[
 Y=\{\rho=0\},
\]
the last two terms vanish.  Hence
\begin{equation}\label{eq:boundaryjetrescale}
 d_J^c\rho_1|_Y
 =
 (e^Hd_J^c\rho)|_Y
 =
 q\alpha,
\end{equation}
because $H|_Y=\log q$ and
\[
 \alpha=d_J^c\rho|_Y.
\]

\smallskip
\noindent\emph{Step 2: strict plurisubharmonicity of $\rho_1$ near $Y$.}
Fix a Hermitian metric on $TX$.  For $p\in Y$, put
\[
 \xi_p=T_pY\cap J(T_pY)
\]
and let
\[
 N_p=\xi_p^\perp.
\]
Since the metric is Hermitian and $\xi_p$ is $J$-invariant, $N_p$ is also
$J$-invariant, and
\[
 T_pX=\xi_p\oplus N_p.
\]
The spaces $N_p$ vary smoothly with $p$.

For a smooth real-valued function $u$ on $X$, denote its Levi quadratic
form by
\[
 \operatorname{Lev}_u:TX\longrightarrow\mathbb R,
 \qquad
 \operatorname{Lev}_u(v)
 :=
 dd_J^cu(v,Jv).
\]
Thus $u$ is strictly $J$-plurisubharmonic precisely when
\[
 \operatorname{Lev}_u(v)>0
\]
for every nonzero $v\in TX$.

Before restricting to $Y$, the Leibniz rule gives
\begin{align}
 dd_J^c(e^H\rho)
 &=
 e^Hdd_J^c\rho
 +e^H\bigl(
 dH\wedge d_J^c\rho
 +d\rho\wedge d_J^cH
 \bigr)
 \notag\\
 &\qquad
 +\rho e^H
 \bigl(
 dd_J^cH+dH\wedge d_J^cH
 \bigr),
 \label{eq:ddceHrho}\\
 dd_J^c(M\rho^2)
 &=
 2M\,d\rho\wedge d_J^c\rho
 +2M\rho\,dd_J^c\rho.
 \label{eq:ddcrho2}
\end{align}
At $p\in Y$, all terms containing an explicit factor of $\rho$ vanish.

For $p\in Y$, define
\[
 G_p(v,w)
 :=
 (dd_J^c\rho_1)_p(v,Jw),
 \qquad
 v,w\in T_pX.
\]
Since $dd_J^c\rho_1$ is a real $(1,1)$-form, $G_p$ is symmetric, and
\[
 \operatorname{Lev}_{\rho_1}(v)=G_p(v,v).
\]
Writing
\[
 v=v_T+v_N,
 \qquad
 v_T\in\xi_p,\quad v_N\in N_p,
\]
we have
\[
 \operatorname{Lev}_{\rho_1}(v)
 =
 G_p(v_T,v_T)
 +2G_p(v_T,v_N)
 +G_p(v_N,v_N).
\]
We estimate these three terms separately.

\smallskip
\noindent\emph{Step 2(a): the tangential block.}
If $v_T\in\xi_p$, then both $v_T$ and $Jv_T$ belong to $T_pY$.  Since
$\rho|_Y=0$,
\[
 d\rho(v_T)=0,
 \qquad
 d_J^c\rho(v_T)=-d\rho(Jv_T)=0.
\]
Thus
\[
 d\rho|_{\xi_p}=0,
 \qquad
 d_J^c\rho|_{\xi_p}=0.
\]
Consequently, the terms
\[
 e^H\bigl(
 dH\wedge d_J^c\rho
 +d\rho\wedge d_J^cH
 \bigr)
\]
in \eqref{eq:ddceHrho}, as well as
\[
 2M\,d\rho\wedge d_J^c\rho,
\]
vanish on $\xi_p\times\xi_p$.  Hence
\[
 G_p(v_T,v_T)
 =
 e^{H(p)}
 (dd_J^c\rho)_p(v_T,Jv_T).
\]
Since $\rho$ is strictly $J$-plurisubharmonic and $e^H>0$, this is
positive for every nonzero $v_T\in\xi_p$.  The unit sphere bundle of
$\xi\to Y$ is compact, so there exists a constant $a_0>0$ such that
\begin{equation}\label{eq:rescale-tangential}
 G_p(v_T,v_T)
 \geq
 a_0\norm{v_T}^2
\end{equation}
for every $p\in Y$ and $v_T\in\xi_p$.

\smallskip
\noindent\emph{Step 2(b): the mixed block.}
For $v_T\in\xi_p$ and $v_N\in N_p$, the identities
\[
 d\rho(v_T)=d_J^c\rho(v_T)=0
\]
imply
\[
 (d\rho\wedge d_J^c\rho)(v_T,Jv_N)=0.
\]
Thus
\[
 2M\,d\rho\wedge d_J^c\rho
\]
does not contribute to the mixed block.  Since $\rho=0$ on $Y$, we have
\[
\begin{aligned}
 G_p(v_T,v_N)
 &=
 e^{H(p)}
 (dd_J^c\rho)_p(v_T,Jv_N)\\
 &\quad
 +e^{H(p)}
 \bigl(
 dH\wedge d_J^c\rho
 +d\rho\wedge d_J^cH
 \bigr)_p(v_T,Jv_N).
\end{aligned}
\]
In particular, this expression is independent of $M$.  It depends
continuously on $(p,v_T,v_N)$ on the compact set
\[
 \left\{
 (p,v_T,v_N)
 \,\middle|\,
 p\in Y,\ 
 v_T\in\xi_p,\ \norm{v_T}=1,\ 
 v_N\in N_p,\ \norm{v_N}=1
 \right\}.
\]
Hence there exists a constant $B\geq0$, independent of both $p$ and $M$,
such that
\begin{equation}\label{eq:rescale-mixed}
 \bigl|G_p(v_T,v_N)\bigr|
 \leq
 B\norm{v_T}\norm{v_N}
\end{equation}
for every $p\in Y$, $v_T\in\xi_p$, and $v_N\in N_p$.

\smallskip
\noindent\emph{Step 2(c): the normal block.}
Put
\[
 Q_p(v_N)
 :=
 (d\rho\wedge d_J^c\rho)_p(v_N,Jv_N).
\]
By the convention
\[
 d_J^c\rho(v)=-d\rho(Jv),
\]
we have
\[
 Q_p(v_N)
 =
 d\rho(v_N)^2+d\rho(Jv_N)^2.
\]
We claim that
\[
 Q_p(v_N)>0
 \qquad
 \text{for every }v_N\in N_p\setminus\{0\}.
\]
Indeed, if $Q_p(v_N)=0$, then
\[
 d\rho(v_N)=d\rho(Jv_N)=0.
\]
Since $Y=\{\rho=0\}$ is a regular level set,
\[
 T_pY=\ker d\rho_p.
\]
Hence
\[
 v_N\in T_pY,
 \qquad
 Jv_N\in T_pY.
\]
The second inclusion implies
\[
 v_N=-J(Jv_N)\in J(T_pY),
\]
so
\[
 v_N\in T_pY\cap J(T_pY)=\xi_p.
\]
Since also $v_N\in N_p$ and
\[
 T_pX=\xi_p\oplus N_p,
\]
we conclude that $v_N=0$, a contradiction.

The unit sphere bundle of $N\to Y$ is compact.  Therefore there exists
$c_0>0$ such that
\begin{equation}\label{eq:rescale-normal-positive}
 Q_p(v_N)
 \geq
 c_0\norm{v_N}^2
\end{equation}
for every $p\in Y$ and $v_N\in N_p$.

The remaining contribution to the normal block is
\[
 (dd_J^c(e^H\rho))_p(v_N,Jv_N).
\]
The function
\[
 (p,v_N)
 \longmapsto
 (dd_J^c(e^H\rho))_p(v_N,Jv_N)
\]
is continuous on the compact unit sphere bundle
\[
 \mathbb S(N|_Y)
 =
 \left\{
 (p,v_N)
 \,\middle|\,
 p\in Y,\ 
 v_N\in N_p,\ 
 \norm{v_N}=1
 \right\}.
\]
Let $m$ be its minimum and set
\[
 C=\max\{0,-m\}.
\]
By homogeneity,
\[
 (dd_J^c(e^H\rho))_p(v_N,Jv_N)
 \geq
 -C\norm{v_N}^2
\]
for every $p\in Y$ and $v_N\in N_p$.  Since
\[
 dd_J^c(M\rho^2)|_Y
 =
 2M\,d\rho\wedge d_J^c\rho,
\]
we obtain
\begin{equation}\label{eq:rescale-normal}
 G_p(v_N,v_N)
 \geq
 (2Mc_0-C)\norm{v_N}^2.
\end{equation}

Combining
\eqref{eq:rescale-tangential},
\eqref{eq:rescale-mixed}, and
\eqref{eq:rescale-normal}, we find
\[
\begin{aligned}
 \operatorname{Lev}_{\rho_1}(v)
 &=
 G_p(v_T,v_T)
 +2G_p(v_T,v_N)
 +G_p(v_N,v_N)\\
 &\geq
 a_0\norm{v_T}^2
 -2B\norm{v_T}\norm{v_N}
 +(2Mc_0-C)\norm{v_N}^2.
\end{aligned}
\]
Completing the square gives
\[
\begin{aligned}
 \operatorname{Lev}_{\rho_1}(v)
 &\geq
 a_0
 \left(
 \norm{v_T}-\frac{B}{a_0}\norm{v_N}
 \right)^2\\
 &\quad+
 \left(
 2Mc_0-C-\frac{B^2}{a_0}
 \right)
 \norm{v_N}^2.
\end{aligned}
\]
Choose $M>0$ so large that
\[
 2Mc_0-C-\frac{B^2}{a_0}>0,
\]
equivalently,
\[
 M>
 \frac{1}{2c_0}
 \left(
 C+\frac{B^2}{a_0}
 \right).
\]
Then
\[
 \operatorname{Lev}_{\rho_1}(v)>0
\]
for every $p\in Y$ and every nonzero $v\in T_pX$.

Fix this value of $M$.  The Levi quadratic form defines a continuous
function
\[
 \operatorname{Lev}_{\rho_1}:TX\longrightarrow\mathbb R.
\]
Its restriction to
\[
 \mathbb S(TX|_Y)
 =
 \{v\in T_pX\mid p\in Y,\ \norm{v}=1\}
\]
is strictly positive.  Since $\mathbb S(TX|_Y)$ is compact, this
restriction has a positive minimum.  Hence, after shrinking the collar of
$Y$, we may assume that
\[
 \operatorname{Lev}_{\rho_1}(v)>0
\]
for every nonzero tangent vector based in that collar.  Thus $\rho_1$ is
strictly $J$-plurisubharmonic on the collar.

\smallskip
\noindent\emph{Step 3: a path of defining functions on a common collar.}
For $0\leq s\leq1$, put
\[
 \rho_s=(1-s)\rho+s\rho_1.
\]
On the collar where both $\rho$ and $\rho_1$ are strictly
$J$-plurisubharmonic, every $\rho_s$ is strictly
$J$-plurisubharmonic, since
\[
 dd_J^c\rho_s
 =
 (1-s)dd_J^c\rho+s\,dd_J^c\rho_1.
\]
Moreover,
\[
 \rho_s
 =
 \rho\bigl((1-s)+se^H+sM\rho\bigr).
\]

Fix the value of $M$ chosen in Step~2.  On $Y\times[0,1]$ we have
\[
 (1-s)+se^H>0.
\]
Since $Y\times[0,1]$ is compact, there is
\[
 m_0:=
 \min_{Y\times[0,1]}
 \bigl((1-s)+se^H\bigr)>0.
\]
By continuity, after shrinking the collar, we may assume that
\[
 (1-s)+se^H\geq\frac{m_0}{2}
\]
throughout the collar for every $s\in[0,1]$.  Since $\rho=0$ on $Y$ and
$M$ is fixed, shrinking the collar once more gives
\[
 M|\rho|<\frac{m_0}{4}.
\]
Hence
\[
 (1-s)+se^H+sM\rho
 \geq
 \frac{m_0}{4}>0
\]
throughout the collar for every $s\in[0,1]$.

Consequently each $\rho_s$ is negative on the interior side, has zero set
exactly $Y$, and satisfies
\[
 d\rho_s|_Y
 =
 \bigl((1-s)+se^H\bigr)d\rho|_Y
 \neq0.
\]
Thus $\{\rho_s\}_{s\in[0,1]}$ is a family of strictly
$J$-plurisubharmonic defining functions on a common boundary collar.

\smallskip
\noindent\emph{Step 4: preparation for the global extension.}
Choose a closed collar
\[
 U_0=\{-\delta_0\leq\rho\leq0\}
\]
on which all the conclusions of Step~3 hold.  Set
\[
 c_*
 :=
 \min_{U_0\times[0,1]}
 \bigl((1-s)+se^H\bigr)>0
\]
and choose
\[
 0<a_1<c_*.
\]
Choose
\[
 0<\delta\leq\delta_0
\]
so small that
\[
 M\delta<\frac{c_*-a_1}{4}.
\]
Replacing $U_0$ by the smaller collar
\[
 U=\{-\delta\leq\rho\leq0\},
\]
choose
\[
 0<b<
 \frac{c_*-a_1}{4}\,\delta.
\]
The function
\[
 a_1\rho-b
\]
is globally defined and strictly $J$-plurisubharmonic.

On the inner boundary $\{\rho=-\delta\}$, we have
\begin{align*}
 (a_1\rho-b)-\rho_s
 &=
 -a_1\delta-b
 +\delta\bigl((1-s)+se^H-sM\delta\bigr)\\
 &=
 \delta\bigl((1-s)+se^H-a_1-sM\delta\bigr)-b.
\end{align*}
Since
\[
 (1-s)+se^H\geq c_*
\]
and
\[
 sM\delta\leq M\delta
 <\frac{c_*-a_1}{4},
\]
we obtain
\begin{align*}
 (a_1\rho-b)-\rho_s
 &\geq
 \delta\bigl(c_*-a_1-sM\delta\bigr)-b\\
 &>
 \delta\left(
 c_*-a_1-\frac{c_*-a_1}{4}
 \right)-b\\
 &=
 \frac{3(c_*-a_1)}{4}\,\delta-b\\
 &>
 \frac{3(c_*-a_1)}{4}\,\delta
 -
 \frac{c_*-a_1}{4}\,\delta\\
 &=
 \frac{c_*-a_1}{2}\,\delta
 >0.
\end{align*}
Thus
\[
 a_1\rho-b>\rho_s
\]
on the inner boundary, uniformly in $s$.

At the outer boundary $Y=\{\rho=0\}$,
\[
 \rho_s-(a_1\rho-b)=b>0.
\]
Since $Y\times[0,1]$ is compact, there is
\[
 0<\delta_{\mathrm{out}}<\delta
\]
such that on
\[
 U_{\mathrm{out}}
 :=
 \{-\delta_{\mathrm{out}}\leq\rho\leq0\}
\]
one has
\[
 \rho_s-(a_1\rho-b)\geq\frac{b}{2}
\]
for every $s\in[0,1]$.

Similarly, compactness of
\[
 \{\rho=-\delta\}\times[0,1]
\]
gives
\[
 0<\delta_{\mathrm{in}}<\delta
\]
such that on
\[
 U_{\mathrm{in}}
 :=
 \{-\delta\leq\rho\leq-\delta+\delta_{\mathrm{in}}\}
\]
one has
\[
 (a_1\rho-b)-\rho_s
 \geq
 \frac{c_*-a_1}{4}\,\delta
\]
for every $s\in[0,1]$.

After shrinking $U_{\mathrm{out}}$ and $U_{\mathrm{in}}$ if necessary, assume
that they are disjoint, and set
\[
 K
 :=
 \overline{
 U\setminus
 \bigl(U_{\mathrm{out}}\cup U_{\mathrm{in}}\bigr)
 }.
\]
Then $K$ is a compact subset of $X\setminus Y$.  Since
\[
 \rho_s<0,
 \qquad
 a_1\rho-b<0
\]
on $K$, the continuous function
\[
 (x,s)
 \longmapsto
 \max\{\rho_s(x),a_1\rho(x)-b\}
\]
is strictly negative on $K\times[0,1]$.  Hence there is $\mu>0$ such that
\[
 \max\{\rho_s,a_1\rho-b\}
 \leq-\mu
\]
on $K$ for every $s\in[0,1]$.

\smallskip
\noindent\emph{Step 5: regularized maximum and globalization.}
Choose
\[
 0<\varepsilon<
 \min\left\{
 \frac{b}{2},
 \frac{c_*-a_1}{4}\,\delta,
 2\mu
 \right\}.
\]
Then
\[
 \rho_s-(a_1\rho-b)>\varepsilon
 \qquad\text{on }U_{\mathrm{out}},
\]
and
\[
 (a_1\rho-b)-\rho_s>\varepsilon
 \qquad\text{on }U_{\mathrm{in}}
\]
for every $s\in[0,1]$.

Choose a smooth even convex function
\[
 \chi_\varepsilon:\mathbb R\longrightarrow\mathbb R
\]
such that
\[
 \chi_\varepsilon(t)=|t|
 \quad\text{for }|t|\geq\varepsilon,
 \qquad
 |\chi_\varepsilon'(t)|\leq1,
 \qquad
 \chi_\varepsilon''(t)\geq0,
\]
and
\[
 0\leq\chi_\varepsilon(t)-|t|\leq\varepsilon
 \qquad
 \text{for all }t\in\mathbb R.
\]
For smooth real-valued functions $f,g$, define their
\emph{regularized maximum} by
\begin{equation}\label{eq:regularizedmax}
 \operatorname{rmax}_\varepsilon(f,g)
 :=
 \frac{f+g+\chi_\varepsilon(f-g)}{2}.
\end{equation}
Since
\[
 \max\{f,g\}
 =
 \frac{f+g+|f-g|}{2},
\]
we have
\[
 \operatorname{rmax}_\varepsilon(f,g)
 =
 \max\{f,g\}
\]
whenever
\[
 |f-g|\geq\varepsilon,
\]
and
\[
 0
 \leq
 \operatorname{rmax}_\varepsilon(f,g)-\max\{f,g\}
 \leq
 \frac{\varepsilon}{2}.
\]

A direct calculation gives
\begin{align*}
 dd_J^c\operatorname{rmax}_\varepsilon(f,g)
 &=
 \frac{1+\chi_\varepsilon'(f-g)}{2}\,dd_J^cf\\
 &\quad+
 \frac{1-\chi_\varepsilon'(f-g)}{2}\,dd_J^cg\\
 &\quad+
 \frac{\chi_\varepsilon''(f-g)}{2}
 d(f-g)\wedge d_J^c(f-g).
\end{align*}
For every $p\in U$ and every $w\in T_pX$,
\[
 \bigl(
 d(f-g)\wedge d_J^c(f-g)
 \bigr)_p(w,Jw)
 =
 d(f-g)_p(w)^2+d(f-g)_p(Jw)^2
 \geq0.
\]
Thus the last term is nonnegative on every complex line.  The first two
coefficients are nonnegative and sum to one.  Hence, if $f$ and $g$ are
strictly $J$-plurisubharmonic, then
\[
 \operatorname{rmax}_\varepsilon(f,g)
\]
is again strictly $J$-plurisubharmonic; compare
\cite{SukhovRegularizedMax}.

For $s\in[0,1]$, define on $U$
\[
 R_s
 :=
 \operatorname{rmax}_\varepsilon
 \bigl(\rho_s,a_1\rho-b\bigr).
\]
By the inequalities above,
\[
 R_s=\rho_s
 \qquad\text{on }U_{\mathrm{out}},
\]
while
\[
 R_s=a_1\rho-b
 \qquad\text{on }U_{\mathrm{in}}.
\]
Moreover, $R_s$ is strictly $J$-plurisubharmonic.

We also verify that no new zero is introduced on the compact middle region
$K$.  There,
\[
 \max\{\rho_s,a_1\rho-b\}\leq-\mu,
\]
and therefore
\[
\begin{aligned}
 R_s
 &\leq
 \max\{\rho_s,a_1\rho-b\}
 +\frac{\varepsilon}{2}\\
 &\leq
 -\mu+\frac{\varepsilon}{2}
 <0,
\end{aligned}
\]
because $\varepsilon<2\mu$.

Since $R_s=a_1\rho-b$ on the fixed inner subcollar
$U_{\mathrm{in}}$, it extends smoothly across the inner edge of $U$ by the
function $a_1\rho-b$.  Define the resulting global function by
\[
 \widehat\rho_s
 =
 \begin{cases}
 R_s,
 & \text{on }U,\\[2mm]
 a_1\rho-b,
 & \text{outside the inner edge of }U.
 \end{cases}
\]
Then $\widehat\rho_s$ is smooth and strictly
$J$-plurisubharmonic on $X$.  Moreover,
\[
 \widehat\rho_s=\rho_s
 \qquad\text{on }U_{\mathrm{out}},
\]
and
\[
 \widehat\rho_s<0
 \qquad\text{on }X\setminus Y.
\]
Thus
\[
 \widehat\rho_s^{-1}(0)=Y,
 \qquad
 d\widehat\rho_s|_Y=d\rho_s|_Y\neq0.
\]
Hence every $\widehat\rho_s$ is a global strictly
$J$-plurisubharmonic defining function for $Y$.

\smallskip
\noindent\emph{Step 6: the global path and the endpoint.}
Since
\[
 \rho_0=\rho,
\]
the functions $\rho$ and $\widehat\rho_0$ agree on
$U_{\mathrm{out}}$.  Both are globally strictly
$J$-plurisubharmonic and negative on $X\setminus Y$.  Therefore
\[
 \sigma_t
 :=
 (1-t)\rho+t\widehat\rho_0,
 \qquad
 0\leq t\leq1,
\]
is a path of global strictly $J$-plurisubharmonic defining functions.
Indeed,
\[
 dd_J^c\sigma_t
 =
 (1-t)dd_J^c\rho
 +t\,dd_J^c\widehat\rho_0
\]
is positive on every nonzero complex line, while
\[
 \sigma_t<0
 \qquad\text{on }X\setminus Y
\]
and
\[
 \sigma_t|_Y=0.
\]
Since $\rho=\widehat\rho_0$ near $Y$, the path $\sigma_t$ is fixed near
the boundary.

After reaching $\widehat\rho_0$, follow the family
\[
 \{\widehat\rho_s\}_{0\leq s\leq1}.
\]
Concatenating these two paths gives a path from $\rho$ to
\[
 \widetilde\rho:=\widehat\rho_1
\]
through global strictly $J$-plurisubharmonic defining functions with zero
set $Y$ and nonvanishing boundary differential throughout.

Finally,
\[
 \widehat\rho_1=\rho_1
\]
on $U_{\mathrm{out}}$, so \eqref{eq:boundaryjetrescale} gives
\[
 d_J^c\widetilde\rho|_Y
 =
 d_J^c\rho_1|_Y
 =
 q\alpha.
\]
This proves both assertions.
\end{proof}

\begin{proposition}[Construction of the normalized Stein data]
\label{prop:normalized-data}
There are strictly $J_\pm$-plurisubharmonic defining functions $\rho_\pm$ and
a diffeomorphism $\tau:Y\to Y$ such that:
\begin{enumerate}[label=\textup{(\roman*)}]
\item for each sign, $\rho_\pm$ is joined to $\rho_0$ through strictly
$J_\pm$-plurisubharmonic defining functions with the same boundary;
\item the final boundary forms are
\[
 \alpha_\pm=\kappa_\pm^*\eta.
\]
In particular, the leaves of the kernel foliation
$\ker(d\alpha_\pm)$ are circles; when oriented by the condition
$\alpha_\pm>0$, they represent the prescribed oriented free-homotopy classes
$[t_\pm]$;
\item $\tau=\kappa_-^{-1}\circ\kappa_+$ is a strict
coorientation-preserving contactomorphism:
\[
 \tau^*\alpha_-=\alpha_+;
\]
\item the final boundary forms define the already fixed cooriented contact
structures:
\[
 \ker\alpha_\pm=\xi_\pm;
\]
\item the exact symplectic volumes agree:
\[
 \int_X\omega_+^2=\int_X\omega_-^2.
\]
\end{enumerate}
Here, 
\[
\lambda_\pm =d^c_{J_{\pm}}\rho_{\pm}, \quad \omega_\pm=d\lambda_\pm, \quad \alpha_{\pm}=\lambda_{\pm}|_{Y}.
\]
\end{proposition}

\begin{proof}
Because $\kappa_\pm$ preserves the coorientation and
$\ker\alpha_\pm^0=\xi_\pm$, there are unique smooth positive functions
$q_\pm:Y\to\R_{>0}$ such that
\[
 \kappa_\pm^*\eta=q_\pm\alpha_\pm^0.
\]
Apply Lemma~\ref{lem:definingrescale} to $(X,J_\pm,\rho_0)$ with the functions
$q_\pm$.  It produces defining functions $\rho_\pm$, together with the paths
asserted in part~\textup{(i)}, for which
\[
 \alpha_\pm=d_{J_\pm}^c\rho_\pm|_Y
 =q_\pm\alpha_\pm^0
 =\kappa_\pm^*\eta.
\]
This proves part~\textup{(ii)}.  Since $\eta$ is a connection form, its
positively oriented characteristic line is tangent to the circle fibers.
Because $\kappa_\pm$ carries $[t_\pm]$ to $[t]$, the pullback
characteristic circles represent $[t_\pm]$.  This proves the final assertion
of part~\textup{(ii)}.

Set $\tau=\kappa_-^{-1}\circ\kappa_+$.  Then
\[
 \tau^*\alpha_-
 =\kappa_+^*(\kappa_-^{-1})^*\kappa_-^*\eta
 =\kappa_+^*\eta
 =\alpha_+,
\]
which proves part~\textup{(iii)}.  Since $q_\pm>0$, the identities
$\alpha_\pm=q_\pm\alpha_\pm^0$ prove part~\textup{(iv)}.

Finally, Stokes' theorem and the strict
boundary equality give
\[
 \int_X\omega_+^2
 =\int_Y\alpha_+\wedge d\alpha_+
 =\int_Y\tau^*(\alpha_-\wedge d\alpha_-)
 =\int_Y\alpha_-\wedge d\alpha_-
 =\int_X\omega_-^2.
\]
This proves part~\textup{(v)}.
\end{proof}

\begin{corollary}[Coincidence after forgetting coorientation]
\label{cor:unoriented-plane-fields}
For the normalized defining functions constructed above,
\[
 \ker(\lambda_+|_Y)=\ker(\lambda_-|_Y)
\]
as unoriented plane subbundles of $TY$.
\end{corollary}

\begin{proof}
One has $\alpha_+=q_+\alpha_+^0$ and
$\alpha_-=q_-\alpha_-^0$ for positive functions $q_\pm$, while
$\alpha_-^0=-\alpha_+^0$ by \eqref{eq:conjugateforms}.  Thus the two kernels
coincide after forgetting coorientation.
\end{proof}

\section{Cap extension and non-symplectomorphism}
\label{sec:cap}\label{sec:nonsymp}
In this section we prove
part~\textup{(iv)} of the Main Theorem~\ref{thm:main}.

\subsection*{Topology of the underlying circle bundle}
\begin{lemma}[Topology of a nontrivial circle bundle]\label{lem:center}
Let $\pi:Y\to\Sigma_g$ be a principal circle bundle of nonzero Euler number $e\in\Z$ over a closed connected oriented surface of genus $g\ge2$.  Let $t$ denote its oriented fiber in
$\pi_1(Y)$ and put $\mathfrak f=[t]\in H_1(Y;\Z)$.  Then:
\begin{enumerate}[label=\textup{(\roman*)}]
\item the fiber map $\pi_1(S^1)\to\pi_1(Y)$ is injective, so $t$ has infinite
order;
\item $H_1(Y;\Z)\cong\Z^{2g}\oplus\Z/|e|\Z$, and $\mathfrak f$ generates the
torsion summand;
\item with standard generators of $\pi_1(\Sigma_g)$ one has
\begin{equation}\label{eq:fundgroupcircle}
 \pi_1(Y)=
 \left\langle
 a_1,b_1,\ldots,a_g,b_g,t
 \ \middle|\
 t\text{ central},\
 \prod_{i=1}^g[a_i,b_i]=t^e
 \right\rangle;
\end{equation}
\item $Z(\pi_1(Y))=\langle t\rangle\cong\Z$.  Consequently every automorphism
of $\pi_1(Y)$ sends $t$ to $t^{\pm1}$.
\end{enumerate}
\end{lemma}

\begin{proof}
The homotopy exact sequence contains
\[
 \pi_2(\Sigma_g)\longrightarrow\pi_1(S^1)
 \longrightarrow\pi_1(Y)\longrightarrow\pi_1(\Sigma_g)\longrightarrow1.
\]
Since $g\geq2$, the base is aspherical and $\pi_2(\Sigma_g)=0$, proving
\textup{(i)}.  The cohomological Gysin sequence contains
\[
 0\longrightarrow H^1(\Sigma_g;\Z)\xrightarrow{\pi^*}H^1(Y;\Z)
 \longrightarrow H^0(\Sigma_g;\Z)
 \xrightarrow{\smile e}H^2(\Sigma_g;\Z).
\]
The final map is multiplication by the nonzero integer $e$ and is injective.
Hence $H^1(Y;\Z)\cong H^1(\Sigma_g;\Z)$ and $b_1(Y)=2g$.  The standard
central-extension presentation is \eqref{eq:fundgroupcircle}; its
abelianization imposes the relation $e[t]=0$ and yields
$H_1(Y;\Z)\cong\Z^{2g}\oplus\Z/|e|\Z$, proving \textup{(ii)} and
\textup{(iii)}.

For \textup{(iv)}, quotienting \eqref{eq:fundgroupcircle} by
$\langle t\rangle$ gives the surface group $\pi_1(\Sigma_g)$, whose center is
trivial.  Every central element of $\pi_1(Y)$ therefore lies in
$\langle t\rangle$, while $t$ is central by the presentation and has infinite
order by \textup{(i)}.  An automorphism restricts to an automorphism of this
infinite cyclic center, and hence sends $t$ to $t^{\pm1}$.
\end{proof}

\subsection*{Mapping-class input for nontrivial circle bundles}

The following two results are stated in the specialized forms used in
this section. Let
$\pi:Y_{g,e}\to\Sigma_g$ denote the oriented circle bundle of nonzero Euler
number $e$ over the closed oriented surface of genus $g\geq2$.

\begin{theorem}[Waldhausen, specialized form]
\label{thm:waldhausen-specialized}
The natural homomorphism
\[
 \pi_0\bigl(\operatorname{Homeo}(Y_{g,e})\bigr)
 \longrightarrow
 \operatorname{Out}\bigl(\pi_1(Y_{g,e})\bigr)
\]
is an isomorphism.\footnote{For a group $G$, an \emph{inner automorphism} is an automorphism of
the form
\[
 c_g:G\longrightarrow G,
 \qquad
 c_g(x)=gxg^{-1}
\]
for some $g\in G$.  The subgroup of all inner automorphisms is denoted by
$\operatorname{Inn}(G)\subset\operatorname{Aut}(G)$.  The
\emph{outer automorphism group} is the quotient
\[
 \operatorname{Out}(G)
 :=
 \operatorname{Aut}(G)/\operatorname{Inn}(G).
\]
Thus an outer automorphism records an automorphism of $G$ only up to
conjugation.  For a self-homeomorphism of a connected space, the induced
automorphism of the fundamental group depends on the choice of base point
and of a path identifying the base point with its image, but changing these
choices changes the induced automorphism only by an inner automorphism.
Hence a self-homeomorphism canonically determines an element of
$\operatorname{Out}(\pi_1)$.}  In particular, two homeomorphisms of $Y_{g,e}$ which
induce the same outer automorphism of the fundamental group are isotopic.
\end{theorem}

\begin{proof}[Proof sketch]
Since the universal cover of $Y_{g,e}$ is homeomorphic to $\mathbb R^3$,
one has
\[
 \pi_2(Y_{g,e})=0.
\]
Hence $Y_{g,e}$ is irreducible by the sphere theorem.
 If $\gamma\subset\Sigma_g$ is an essential simple
closed curve, then the vertical torus $\pi^{-1}(\gamma)$ is incompressible:
its fundamental group $\mathbb Z^2$ injects into the central extension
$\pi_1(Y_{g,e})$.  Hence $Y_{g,e}$ is sufficiently large.\footnote{In Waldhausen's terminology, a compact $3$-manifold is called
\emph{sufficiently large} if it contains a properly embedded two-sided
incompressible surface.  Here a surface $F\subset M$ is
\emph{incompressible} if the inclusion induces an injection
\[
 \pi_1(F)\longrightarrow\pi_1(M).
\]
In the present closed orientable case, the vertical torus
$\pi^{-1}(\gamma)\subset Y_{g,e}$ provides such a surface.}
  Waldhausen's
hierarchy theorem applies to such irreducible sufficiently large
$3$--manifolds; Corollary~7.5 of \cite{Waldhausen} identifies mapping classes
with outer automorphisms, and the remark following that corollary gives the
corresponding isotopy statement.  The full argument cuts along an
incompressible hierarchy, compares the induced maps on the successive
pieces, and then reassembles the isotopies.
\end{proof}

\begin{theorem}[Chen--Tshishiku, specialized form]
\label{thm:chen-tshishiku-specialized}
Let $z$ be the oriented generator of the center of
$\pi_1(Y_{g,e})$. 
Let
\[
 \mathcal{OUT}_{\mathrm{bf}}\bigl(\pi_1(Y_{g,e})\bigr)
<
 \operatorname{Out}\bigl(\pi_1(Y_{g,e})\bigr)
\]
be the subgroup consisting of outer automorphisms which fix $z$ and induce
an orientation-preserving outer automorphism of
\[
 \pi_1(Y_{g,e})/\langle z\rangle
 \cong
 \pi_1(\Sigma_g).
\]  Then there is a short
exact sequence
\[
 1\longrightarrow H^1(\Sigma_g;\mathbb Z)
 \longrightarrow
 \mathcal{OUT}_{\mathrm{bf}}(\pi_1(Y_{g,e}))
 \longrightarrow
 \operatorname{Mod}(\Sigma_g)
 \longrightarrow1.
\]
Moreover, every element of
$\mathcal{OUT}_{\mathrm{bf}}(\pi_1(Y_{g,e}))$ is induced by an
orientation-preserving fiber-preserving homeomorphism of $Y_{g,e}$.
\end{theorem}

\begin{proof}[Proof sketch]
The fundamental group is the central extension
\[
 1\longrightarrow\langle z\rangle
 \longrightarrow\pi_1(Y_{g,e})
 \longrightarrow\pi_1(\Sigma_g)
 \longrightarrow1.
\]
Because the center is precisely $\langle z\rangle$, every relevant
automorphism descends to the surface group.  The kernel consists of the
transvections
\[
 \gamma\longmapsto \gamma z^{\phi(\bar\gamma)},
 \qquad
 \phi\in\operatorname{Hom}(\pi_1(\Sigma_g),\mathbb Z)
       =H^1(\Sigma_g;\mathbb Z).
\]
The Dehn--Nielsen--Baer theorem identifies the orientation-preserving outer
automorphisms of the quotient with $\operatorname{Mod}(\Sigma_g)$.
Chen--Tshishiku construct the resulting exact sequence and its geometric
realization in Sections~2.2--2.3 of \cite{ChenTshishiku}; the transvections
are realized by vertical fiber twists, while a base mapping class lifts to a
fiber-preserving homeomorphism.  Waldhausen's theorem above identifies these
geometric representatives with the required mapping classes.
\end{proof}
\subsection*{Cap extension}
\begin{proposition}[Cap extension]\label{prop:capextension}
For $i=1,2$, let $E_i\to\Sigma_i$ be oriented disk bundles over closed
oriented surfaces of genus at least two, let $Y_i=\partial E_i$, and orient each
circle fiber $u_i$ as the boundary of the oriented disk fiber.  Suppose the
oriented circle bundles have the same nonzero Euler number.  If
\[
 h:Y_1\longrightarrow Y_2
\]
is an orientation-preserving homeomorphism satisfying
$h_*[u_1]=[u_2]$ as oriented free-homotopy classes, then $h$ is isotopic,
through orientation-preserving homeomorphisms, to a fiber-preserving
homeomorphism $h_0$ which preserves the fiber and base orientations.
Moreover, $h_0$ extends to an orientation-preserving
homeomorphism
\[
 \widehat h_0:E_1\longrightarrow E_2
\]
which maps the zero section to the zero section and preserves its orientation.
\end{proposition}

\begin{proof}
The Gysin sequence for the oriented
circle bundle
\[
 \pi_i:Y_i\longrightarrow\Sigma_i
\]
gives
\[
 0
 \longrightarrow
 H^1(\Sigma_i;\Z)
 \xrightarrow{\ \pi_i^*\ }
 H^1(Y_i;\Z)
 \xrightarrow{\ (\pi_i)_*\ }
 H^0(\Sigma_i;\Z)
 \xrightarrow{\ \smile e_i\ }
 H^2(\Sigma_i;\Z),
\]
where $e_i\in H^2(\Sigma_i;\Z)$ is the Euler class.  Since $\Sigma_i$ is
closed, connected, and oriented,
\[
 H^0(\Sigma_i;\Z)\cong\Z,
 \qquad
 H^2(\Sigma_i;\Z)\cong\Z,
\]
and the last map is multiplication by the nonzero Euler number
$\langle e_i,[\Sigma_i]\rangle$.  It is therefore injective.  Exactness
then implies that $(\pi_i)_*$ has zero image, and hence that $\pi_i^*$ is
surjective.  Since $\pi_i^*$ is already injective by exactness at
$H^1(\Sigma_i;\Z)$, it follows that
\[
 \pi_i^*:H^1(\Sigma_i;\Z)
 \xrightarrow{\cong}
 H^1(Y_i;\Z).
\]
Hence $b_1(Y_i)=2g(\Sigma_i)$.
Since $h$ induces an isomorphism
\[
 h^*:H^1(Y_2;\Z)\xrightarrow{\cong}H^1(Y_1;\Z),
\]
we have
\[
 2g(\Sigma_1)
 =
b_1(Y_1)
 =
 b_1(Y_2)
 =
 2g(\Sigma_2),
\]
and hence
\[
 g(\Sigma_1)=g(\Sigma_2).
\]
Choose an orientation-preserving diffeomorphism
\[
 f:\Sigma_1\longrightarrow\Sigma_2.
\]
Since the two oriented circle bundles have the same
Euler number and $f$ preserves orientation,
\[
 \bigl\langle f^*e_2,[\Sigma_1]\bigr\rangle
 =
 \bigl\langle e_2,[\Sigma_2]\bigr\rangle
 =
 \bigl\langle e_1,[\Sigma_1]\bigr\rangle.
\]
Since
\[
 H^2(\Sigma_1;\Z)\cong\Z,
\]
it follows that
\[
 f^*e_2=e_1.
\]
Oriented circle bundles over $\Sigma_1$ are classified up to
fiber-orientation-preserving bundle isomorphism by their Euler classes.
Hence there exists a fiber-orientation-preserving bundle isomorphism
\[
 b:Y_1\longrightarrow Y_2
\]
covering $f$. Set 
\[
\varphi=b^{-1}\circ h: Y_1\to Y_1.
\]

The oriented fiber subgroup is the center of $\pi_1(Y_1)$.  The hypothesis
$h_*[u_1]=[u_2]$ therefore says that the outer automorphism induced by
$\varphi$ fixes the oriented central generator. 
The oriented fiber subgroup
\[
 \langle z\rangle\cong\mathbb Z
\]
is the center of $\pi_1(Y_1)$, and the quotient fits into the central
extension
\[
 1\longrightarrow \langle z\rangle
 \longrightarrow \pi_1(Y_1)
 \longrightarrow \pi_1(\Sigma_1)
 \longrightarrow 1.
\]
Its extension class
\[
 e\in H^2(\pi_1(\Sigma_1);\Z)
 \cong H^2(\Sigma_1;\Z)
\]
is the Euler class of the oriented circle bundle
$Y_1\to\Sigma_1$.

The hypothesis on the oriented fiber class implies that the outer
automorphism induced by $\varphi$ fixes the chosen generator $z$ of the
center.  Hence it induces an outer automorphism
\[
 \overline{\varphi}_*
 \in
 \operatorname{Out}\bigl(\pi_1(\Sigma_1)\bigr).
\]
  Since $\varphi_*$ acts trivially on the central
coefficient group $\langle z\rangle\cong\Z$, naturality of the extension
class gives
\[
 \overline{\varphi}_*^{\,*}e=e.
\]

Now
\[
 H^2(\Sigma_1;\Z)\cong\Z,
\]
and an orientation-preserving outer automorphism of the surface group acts
as $+1$ on this group, whereas an orientation-reversing one acts as $-1$.
If $\overline{\varphi}_*$ were orientation reversing, we would therefore
have
\[
 \overline{\varphi}_*^{\,*}e=-e.
\]
Together with
\[
 \overline{\varphi}_*^{\,*}e=e
\]
this would imply $e=-e$, which is impossible because the Euler class is
nonzero.  Hence $\overline{\varphi}_*$ preserves the orientation of the
base.  Thus $[\varphi_*]$ fixes the oriented central generator and induces
an orientation-preserving outer automorphism of the quotient, so
\[
 [\varphi_*]
 \in
 \mathcal{OUT}_{\mathrm{bf}}\bigl(\pi_1(Y_1)\bigr).
\]

By Theorem~\ref{thm:chen-tshishiku-specialized}, there is an
orientation-preserving fiber-preserving homeomorphism
$\varphi_0:Y_1\to Y_1$ inducing the same outer automorphism as $\varphi$.
Theorem~\ref{thm:waldhausen-specialized} then implies that $\varphi$ and
$\varphi_0$ are isotopic.  Consequently $h$ is isotopic to the
fiber-preserving homeomorphism
\[
 h_0=b\circ \varphi_0:Y_1\longrightarrow Y_2.
\]
The action on the oriented center is invariant under isotopy, so $h_0$
preserves the circle-fiber orientation.  Let
$\bar h_0:\Sigma_1\to\Sigma_2$ be the induced base map.  Naturality gives
\[
 \bar h_0^*e_2=e_1.
\]
Since the evaluated Euler numbers are the same nonzero integer, this forces
$\deg(\bar h_0)=+1$.  Hence the base orientation is preserved.

Choose fiberwise norms on the disk bundles.  They identify each $E_i$ with
the fiberwise cone on $Y_i$:
\[
 E_i\cong\bigl(Y_i\times[0,1]\bigr)/\!\sim,
\]
where every circle $Y_{i,x}\times\{0\}$ is collapsed to the zero vector over
$x$.  Define
\[
 \widehat h_0([y,r])=[h_0(y),r].
\]
This is well defined, continuous at the cone points, and invertible.  On each
disk fiber it is the Alexander cone extension of an
orientation-preserving circle homeomorphism.  It therefore preserves the
disk-fiber orientation, maps the zero section to the zero section, and
restricts there to the orientation-preserving base map.  Hence it preserves
the total orientation and the orientation of the zero section.
\end{proof}

\begin{unnumberedremark}
Proposition~\ref{prop:capextension} is used in both parts~\textup{(iv)} and
\textup{(v)} of Theorem~\ref{thm:main}, but the hypothesis
$h_*[u_1]=[u_2]$ is obtained in two different ways.  In the
symplectic argument, the boundary restriction of the hypothetical
symplectomorphism preserves the oriented characteristic line of the
normalized endpoint forms, so Lemma~\ref{lem:symp-characteristic} gives the
fiber condition directly.  In the Weinstein argument, the boundary map
obtained from Gray stability is only a contactomorphism of the cooriented
plane fields; the required fiber condition is instead supplied by the
Floer-theoretic Theorem~\ref{thm:fiberdetection}.  The topological extension
step is the same in both arguments; only the source of its fiber-preservation
hypothesis differs.
\end{unnumberedremark}

In our divisor-complement situation, the boundary-fiber orientations satisfy
\[
 u_\pm=t_\pm^{-1}
\]
by \eqref{eq:orientationtable}.  Reversing the oriented fiber changes the sign
of the Euler class, so the $u_\pm$-oriented cap-boundary circle bundles over
$D_\pm$ both have Euler number $+36$.  The contact orientation on
$Y=\partial X$ is the convex-boundary orientation, whereas the orientation of
the same hypersurface as $\partial\nu(D_\pm)$ is its negative.   Consequently, any orientation-preserving
boundary homeomorphism carrying $[t_+]$ to $[t_-]$ also carries
$[u_+]=[t_+^{-1}]$ to $[u_-]=[t_-^{-1}]$ and satisfies
Proposition~\ref{prop:capextension} when regarded as a map between the cap
boundaries.

\subsection*{Non-symplectomorphism}
The following collar case of topological isotopy extension
is used in the capping arguments to prove parts (iv) and (v) of the Main Theorem~\ref{thm:main}.
 \begin{lemma}[Topological collar adjustment]
\label{lem:collar-adjustment}
Let $F:X_1\to X_2$ be a homeomorphism of compact manifolds with boundary, and
let $k_t:\partial X_1\to\partial X_2$, $0\le t\le1$, be an isotopy with
$k_0=F|_{\partial X_1}$.  Then $F$ is isotopic, through homeomorphisms supported
in a boundary collar, to a homeomorphism $F_1$ with
$F_1|_{\partial X_1}=k_1$.
\end{lemma}

\begin{proof}
Choose a product collar
\[
 c_2:\partial X_2\times[0,\varepsilon]
 \longrightarrow X_2
\]
of $\partial X_2$.  Define a collar of $\partial X_1$ by
\[
 c_1(y,r)
 :=
 F^{-1}\bigl(c_2(k_0(y),r)\bigr).
\]
Then
\[
 F\bigl(c_1(y,r)\bigr)
 =
 c_2(k_0(y),r),
\]
so, in these collar coordinates,
\[
 F(y,r)=(k_0(y),r).
\]

Choose a smooth function
\[
 \chi:[0,\varepsilon]\longrightarrow[0,1]
\]
such that
\[
 \chi(r)=1
 \quad\text{for }r\text{ near }0,
 \qquad
 \chi(r)=0
 \quad\text{for }r\text{ near }\varepsilon.
\]
Define
\[
 F_1\bigl(c_1(y,r)\bigr)
 :=
 c_2\bigl(k_{\chi(r)}(y),r\bigr)
\]
on the collar, and set
\[
 F_1=F
\]
outside the collar.  Since $\chi=0$ near the inner edge \(\partial X_1\times\{\varepsilon\}\),
\[
 F_1\bigl(c_1(y,r)\bigr)
 =
 c_2(k_0(y),r)
 =
 F\bigl(c_1(y,r)\bigr)
\]
there, so the two definitions glue to a global homeomorphism
\[
 F_1:X_1\longrightarrow X_2.
\]
Since $\chi=1$ near $r=0$, its boundary restriction is
\[
 F_1|_{\partial X_1}=k_1.
\]

Finally, for $0\leq s\leq1$, define on the collar
\[
 F_s\bigl(c_1(y,r)\bigr)
 :=
 c_2\bigl(k_{s\chi(r)}(y),r\bigr),
\]
and put
\[
 F_s=F
\]
outside the collar.  Again the definitions agree near the inner edge \(\partial X_1\times\{\varepsilon\}\) because
$\chi=0$ there.  Thus $\{F_s\}_{0\leq s\leq1}$ is an isotopy through
homeomorphisms from $F_0=F$ to $F_1$, and
\[
 F_s|_{\partial X_1}=k_s.
\]
\end{proof}

\begin{lemma}[Oriented characteristic direction]
\label{lem:symp-characteristic}
Let $(X_i,\omega_i=d\lambda_i)$ be compact exact symplectic four-manifolds with
convex boundary $Y_i$, and put $\alpha_i=\lambda_i|_{Y_i}$.  If
\[
 F:(X_1,\omega_1)\longrightarrow(X_2,\omega_2)
\]
is a symplectomorphism, then $F|_{Y_1}$ carries the characteristic line
$\ker(d\alpha_1)$ to $\ker(d\alpha_2)$ and preserves its orientation, where
that line is oriented by the condition $\alpha_i>0$.  In particular, if these oriented characteristic lines integrate to circle foliations whose
oriented leaves represent $[t_i]$, then
\[
 (F|_{Y_1})_*[t_1]=[t_2].
\]
\end{lemma}

\begin{proof}
Set
\[
 \beta=(F|_{Y_1})^*\alpha_2.
\]
Since $F$ is symplectic,
\[
 d\beta=(F|_{Y_1})^*(\omega_2|_{Y_2})
       =\omega_1|_{Y_1}=d\alpha_1.
\]
The map $F|_{Y_1}$ preserves the boundary orientation, and hence both
$\alpha_1\wedge d\alpha_1$ and $\beta\wedge d\alpha_1$ are positive volume
forms on $Y_1$.  Let $v$ be a positive generator of
$\ker(d\alpha_1)$, so $\alpha_1(v)>0$. 
Choose
$w_1,w_2\in TY_1$ such that
\[
 d\alpha_1(w_1,w_2)>0.
\]
Since $v\in\ker(d\alpha_1)$, we have
\[
 d\alpha_1(v,\cdot)=0.
\]
Therefore
\[
 (\beta\wedge d\alpha_1)(v,w_1,w_2)
 =
 \beta(v)\,d\alpha_1(w_1,w_2).
\]
The forms $\beta\wedge d\alpha_1$ and
$\alpha_1\wedge d\alpha_1$ determine the same positive orientation on
$Y_1$, so the left-hand side is positive.  Since
$d\alpha_1(w_1,w_2)>0$, it follows that
\[
 \beta(v)>0.
\]

We next show that $F_*v \in \ker (d\alpha_2)$.  For
any $w\in T_{F(p)}Y_2$, choose $u\in T_pY_1$ such that
\[
 w=F_*u.
\]
Using
\[
 d\alpha_i=\omega_i|_{TY_i}
 \qquad\text{and}\qquad
 F^*\omega_2=\omega_1,
\]
we obtain
\[
\begin{aligned}
 d\alpha_2(F_*v,w)
 &=
 d\alpha_2(F_*v,F_*u)\\
 &=
 \omega_2(F_*v,F_*u)\\
 &=
 \omega_1(v,u)\\
 &=
 d\alpha_1(v,u)
 =0.
\end{aligned}
\]
Hence
\[
 F_*v\in\ker(d\alpha_2).
\]
Moreover, since $\beta=F^*\alpha_2$,
\[
 \alpha_2(F_*v)
 =
 (F^*\alpha_2)(v)
 =
 \beta(v)
 >0.
\]
Thus $F_*$ carries the positively oriented characteristic line of $Y_1$
to the positively oriented characteristic line of $Y_2$.
\end{proof}

\begin{theorem}[Non-symplectomorphism]\label{thm:nonsymp}
For the normalized Stein data of
Proposition~\ref{prop:normalized-data},
\[
 (X,\omega_+)\not\cong_{\mathrm{symp}}(X,\omega_-).
\]
\end{theorem}

\begin{proof}
Suppose that
\[
 F:(X,\omega_+)\longrightarrow(X,\omega_-)
\]
is a symplectomorphism.  By
Proposition~\ref{prop:normalized-data}\textup{(ii)}, the positively oriented characteristic circles of
$\omega_\pm|_Y=d\alpha_\pm$ represent the classes $[t_\pm]$.  Lemma~\ref{lem:symp-characteristic} therefore gives
\[
 (F|_Y)_*[t_+]=[t_-].
\]
Since $u_\pm=t_\pm^{-1}$ by \eqref{eq:orientationtable}, the same boundary map
preserves the oriented disk-boundary fibers $u_\pm$.  Regard $F|_Y$ as a map
from $\partial\nu(D_+)$ to $\partial\nu(D_-)$.  The simultaneous reversal from
the convex-boundary orientation to the cap-boundary orientation occurs on both
sides, so this map is orientation preserving.  Proposition~\ref{prop:capextension}
then gives an isotopy from $F|_Y$ to a fiber-preserving homeomorphism $h_0$ and
an orientation-preserving cap map
\[
 \widehat h_0:\nu(D_+)\longrightarrow\nu(D_-)
\]
which maps the oriented zero section $D_+$ to the oriented zero section $D_-$.

Apply Lemma~\ref{lem:collar-adjustment} to the underlying homeomorphism of $F$
and this boundary isotopy.  We obtain an orientation-preserving homeomorphism
$F_0:X\to X$, equal to $F$ away from a collar, with $F_0|_Y=h_0$.  The full extension over the normal disk-bundle cap is therefore
\[
 \overline h_0
 =F_0\cup_{h_0}\widehat h_0:
 X\cup_Y\nu(D_+)\longrightarrow X\cup_Y\nu(D_-).
\]
Both source and target
are the fake projective plane $S$.  Thus $\overline h_0:S\to S$ is an
orientation-preserving self-homeomorphism carrying the oriented zero section
$D_+$ to $D_-$.

The divisor is bicanonical, so
\[
 \PD[D_+]=2K.
\]
For the conjugate complex structure the complex orientation of $D$ reverses,
and therefore
\[
 \PD[D_-]=-2K
\]
inside the same oriented four-manifold $S$.  Since
$(\overline h_0)_*[D_+]=[D_-]$, Poincar\'e duality gives
\begin{equation}\label{eq:capclassrelation}
 \overline h_0^{*}(-2K)=2K.
\end{equation}
On the other hand, Proposition~\ref{prop:KKsurface} gives
$\overline h_0^{*}K=K$.  Equation~\eqref{eq:capclassrelation} would imply
$-2K=2K$, hence $4K=0$, contradicting $K^2=9$.
\end{proof}

The proof used only the oriented characteristic foliation of the normalized
boundary forms and the Waldhausen--Chen--Tshishiku cap-extension input.  In
particular, neither monopole Floer homology nor ECH enters
Theorem~\ref{thm:nonsymp}.

\section{Monopole Floer homology detects the positive fiber class}\label{sec:monopole}

This section supplies the additional boundary rigidity needed only for
part~\textup{(v)} of Theorem~\ref{thm:main}.  
Let
\[
 \pi:Y\longrightarrow\Sigma_g
\]
be an oriented principal circle bundle of Euler number $-n<0$ over a closed oriented surface of genus $g$.  Let $t$ denote
the positive fiber in $\pi_1(Y)$ and set
\[
 \mathfrak f=[t]\in H_1(Y;\Z).
\]
Let $\xi$ be the positive prequantization contact structure and
$\mathfrak{s}_\xi$ its $\SpinC$ structure.  For $r\in\Z/n$, define
\[
 \mathfrak{s}_{\xi,r}=\mathfrak{s}_\xi+\PD(r\mathfrak f).
\]
\begin{unnumberedlemma}[Torsion of the relevant first Chern classes]
For every $r\in\Z/n$, the classes
\[
 c_1(\xi),
 \qquad
 \PD(\mathfrak f),
 \qquad
 c_1(\mathfrak{s}_{\xi,r})
\]
are torsion.  More precisely, if
$\mu\in H^2(\Sigma_g;\Z)$ is the positive generator, then
\[
 n\,\pi^*\mu=0,
 \qquad
 c_1(\xi)=(2-2g)\pi^*\mu,
\]
and
\begin{equation}\label{eq:spinc-c1-torsion}
 c_1(\mathfrak{s}_{\xi,r})
 =c_1(\xi)+2\PD(r\mathfrak f).
\end{equation}
\end{unnumberedlemma}

\begin{proof}
Since the Boothby--Wang contact structure is the horizontal distribution of
the connection, the restriction of the bundle projection gives a natural
isomorphism
\[
 d\pi|_\xi:\xi\xrightarrow{\cong}\pi^*T\Sigma_g.
\]
Hence
\[
 c_1(\xi)=\pi^*c_1(T\Sigma_g).
\]
The relevant part of the cohomological Gysin sequence is
\[
 H^0(\Sigma_g;\Z)
 \xrightarrow{\ \smile e\ }
 H^2(\Sigma_g;\Z)
 \xrightarrow{\ \pi^*\ }
 H^2(Y;\Z),
\]
where $e\in H^2(\Sigma_g;\Z)$ is the Euler class of the circle bundle.
Let $\mu\in H^2(\Sigma_g;\Z)$ be the generator satisfying
\[
 \langle\mu,[\Sigma_g]\rangle=1.
\]
Since the Euler number is $-n$, one has
\[
 e=-n\mu.
\]
Exactness therefore gives
\[
 \ker\pi^*=n\Z\,\mu,
\]
and in particular
\[
 n\,\pi^*\mu=0.
\]
Since
\[
 c_1(T\Sigma_g)=(2-2g)\mu,
\]
it follows that
\[
 c_1(\xi)=(2-2g)\pi^*\mu
\]
is torsion.

By Lemma~\ref{lem:center}\textup{(ii)}, the fiber class
$\mathfrak f$ generates the torsion summand $\Z/n\subset H_1(Y;\Z)$.
Thus $\PD(\mathfrak f)$ is torsion.  Finally, the affine action of
$H^2(Y;\Z)$ on $\SpinC(Y)$ satisfies
\[
 c_1(\mathfrak s+a)=c_1(\mathfrak s)+2a.
\]
Since $c_1(\mathfrak s_\xi)=c_1(\xi)$ and
$\mathfrak{s}_{\xi,r}=\mathfrak s_\xi+\PD(r\mathfrak f)$, this gives
\eqref{eq:spinc-c1-torsion}.  It is a sum of torsion classes and is therefore
torsion.
\end{proof}

Therefore the monopole Floer groups used below carry relative
$\Z$-gradings.  We use the  $\mathbb{F}_2$ coefficient version of the monopole Floer cohomology group denoted
$\widehat{HM}^{\bullet}$ in \cite{KronheimerMrowka}, with contravariant
diffeomorphism maps
\[
 h^*: \widehat{HM}^{\bullet}(Y_2,\mathfrak s_2)
 \longrightarrow
 \widehat{HM}^{\bullet}(Y_1,h^*\mathfrak s_2)
\]
for $h: Y_1 \to Y_2$.
For a torsion $\SpinC$ structure, an orientation-preserving diffeomorphism
preserves relative integer grading differences, without reversing their sign;
an overall translation of an absolute lift is irrelevant here.  The induced map is the
cobordism map of the mapping cylinder.  For the product $\SpinC$ structure,
the Euler characteristic and signature vanish, and the torsion first Chern
class has square zero.  The standard cobordism degree-shift formula therefore
gives shift zero.  This is the functoriality developed in
\cite[Chapter~VII, Section~25]{KronheimerMrowka}, together with the relative
and canonical grading conventions of
\cite[Chapter~VIII, Section~28]{KronheimerMrowka}.

We write $\HM^{-*}$ for the same cohomology group with its cohomological
grading reindexed by $k\mapsto-k$.  Thus Taubes's isomorphism is grading
preserving when its target is written as $\HM^{-*}$.

Two relatively $\Z$-graded vector spaces are called isomorphic if there is a
linear isomorphism which preserves all grading differences; an overall
translation of the grading is allowed, but a sign reversal is not.  The ECH
groups below are bounded below and have a unique lowest class in the two
homology classes of interest.  Their $\HM^{-*}$ counterparts have the same
ordering, whereas the unreversed monopole Floer grading has a unique highest
class.  We write $\ECH_*^{\rm norm}$ for the ECH grading translated so that
the unique lowest class has degree zero.

The calculation is most explicit on the ECH side, where Nelson--Weiler give
both generators and relative degrees.  In principle, contactomorphism
naturality of ECH would already suffice for the fiber-class obstruction.
We pass through Taubes's isomorphism in order to formulate the obstruction in
Seiberg--Witten Floer theory and to use the standard diffeomorphism naturality
of $\widehat{HM}$.  A proof using only $\widehat{HM}$ would require a direct
monopole Floer computation of the same relative grading profiles; no such
independent computation is attempted here.  The relative $\Z$-grading used
below remains the one justified above by the torsion of
$c_1(\mathfrak s_{\xi,r})$.

\begin{theorem}[Nelson--Weiler and Taubes]\label{thm:HMgradedformula}
Fix an integer representative $0\leq r\leq n-1$ of the class $r\mathfrak f$.  There is a
relatively $\Z$-graded isomorphism
\begin{equation}\label{eq:ECH-HM-isom}
 \ECH_*(Y,\xi,r\mathfrak f;\F)
 \cong
 \HM^{-*}(Y,\mathfrak{s}_\xi+\PD(r\mathfrak f);\F).
\end{equation}
As a vector space, the left-hand side is
\begin{equation}\label{eq:gradeddecomposition}
 \bigoplus_{d\geq0}\mathcal{A}_{r+nd},
\end{equation}
where $\mathcal{A}_M$ denotes the graded vector space identified by
Nelson--Weiler with the corresponding ECH homology group and has a basis of
monomials
\begin{equation}\label{eq:monomial}
 e_-^{m_-}h_{i_1}\cdots h_{i_j}e_+^{m_+},
 \qquad
 m_-+j+m_+=M,
 \qquad
 1\leq i_1<\cdots<i_j\leq2g.
\end{equation}
By the ECH admissibility condition, the elliptic generators $e_-$ and $e_+$ may occur with arbitrary
multiplicity, whereas each hyperbolic generator $h_i$ occurs at most once.
These monomials describe a basis of the computed homology group, not merely
of an uncomputed chain complex.  Every fixed internal degree contains only
finitely many such monomials.  Put
\[
 q=j+2m_+.
\]
If the reference monomial $e_-^r\in\mathcal{A}_r$ is assigned ECH degree zero
(with $e_-^0=1$ the empty orbit-set generator), then a monomial in the summand
$\mathcal{A}_{r+nd}$ has relative ECH degree
\begin{equation}\label{eq:relative-degree}
 I_r(d,q)=nd^2+(\chi(\Sigma_g)+2r-n)d+q.
\end{equation}
\end{theorem}

\begin{proof}
We spell out the translation because every sign will be used later.  Nelson and
Weiler take the Euler class of a prequantization bundle to be the negative
integer $e<0$ and prove
\begin{equation}\label{eq:NW-decomposition}
 \ECH_*(Y,\xi,\Gamma;\F)
 \cong
 \bigoplus_{d\geq0}
 \Lambda^{\Gamma+(-e)d}H_*(\Sigma_g;\F).\footnote{Nelson--Weiler use the symbol $\Gamma$ in two closely related
senses.  In $\ECH_*(Y,\xi,\Gamma;\F)$ it denotes a homology class
$\Gamma\in H_1(Y;\Z)$.  For the prequantization bundle considered here,
the relevant classes are torsion fiber classes
\[
 \Gamma=r\mathfrak f,
 \qquad
 r\in\Z/(-e),
\]
and, after choosing the representative
$0\leq r\leq -e-1$, the same symbol $\Gamma$ is also used for this integer
representative.  Thus the exponent in
\[
 \Lambda^{\Gamma+(-e)d}H_*(\Sigma_g;\F)
\]
is an integer, not a homology class.  In our notation we avoid this
ambiguity by writing the exponent as $r+(-e)d$, and hence as $r+nd$ when
$e=-n$.}
\end{equation}
For our convention \eqref{eq:eulerconvention}, a positive-curvature connection
has Euler number $e=-n$.  The fiber class has order $n$, and for the torsion
class $r\mathfrak f$ we choose the representative $\Gamma=r\in\{0,\ldots,n-1\}$.
Thus the multiplicity in \eqref{eq:NW-decomposition} is
\[
 \Gamma+(-e)d=r+nd.
\]
Their graded exterior-power notation is the object described by
\eqref{eq:monomial}: the degree-zero and degree-two elliptic generators may be
repeated, while ECH admissibility allows each positive hyperbolic generator at
most once.

Nelson--Weiler's relative grading formula \cite[Theorem 1.1, formula
(1.4)]{NelsonWeiler} is
\begin{align}
 |\alpha|_{\mathrm{ECH}}-|\beta|_{\mathrm{ECH}}
 &=-e(d_\alpha^2-d_\beta^2)
 +(\chi(\Sigma_g)+2\Gamma+e)(d_\alpha-d_\beta)
 \notag\\
 &\qquad+|\alpha|_\bullet-|\beta|_\bullet.
 \label{eq:NW-original-grading}
\end{align}
\footnote{Here $d_\alpha,d_\beta\in\Z_{\geq0}$ record the summands in the
Nelson--Weiler decomposition.  More precisely, if $\gamma$ is either
$\alpha$ or $\beta$ and belongs to
\[
 \Lambda^{\,\Gamma+(-e)d_\gamma}H_*(\Sigma_g;\F),
\]
then $d_\gamma$ is the corresponding nonnegative integer.  Equivalently,
if
\[
 \gamma
 =
 e_-^{m_-}h_{i_1}\cdots h_{i_j}e_+^{m_+}
\]
has total multiplicity
\[
 M_\gamma=m_-+j+m_+,
\]
then
\[
 M_\gamma=\Gamma+(-e)d_\gamma,
 \qquad
 d_\gamma=\frac{M_\gamma-\Gamma}{-e}.
\]
Here $\Gamma$ denotes the chosen integer representative of the relevant
torsion fiber class, rather than the homology class itself.

The symbol $|\gamma|_\bullet$ denotes the internal grading coming from the
ordinary homological grading on $H_*(\Sigma_g;\F)$: the generator $e_-$
corresponding to $H_0(\Sigma_g;\F)$ has degree $0$, each hyperbolic
generator $h_i$ corresponding to $H_1(\Sigma_g;\F)$ has degree $1$, and
the generator $e_+$ corresponding to $H_2(\Sigma_g;\F)$ has degree $2$.
Thus
\[
 |\gamma|_\bullet=j+2m_+.
\]
In particular,
\[
 |\alpha|_\bullet-|\beta|_\bullet
\]
is the difference of these internal homological degrees.}
Take $\beta=e_-^r$.  Then $d_\beta=0$ and
$|\beta|_\bullet=0$.  For a monomial $\alpha$ in
$\mathcal{A}_{r+nd}$ one has $d_\alpha=d$ and
$|\alpha|_\bullet=q=j+2m_+$.  Substituting $e=-n$ and $\Gamma=r$ into
\eqref{eq:NW-original-grading} gives
\[
 |\alpha|_{\mathrm{ECH}}-|e_-^r|_{\mathrm{ECH}}
 =nd^2+(\chi(\Sigma_g)+2r-n)d+q,
\]
which is \eqref{eq:relative-degree}.

For the $\SpinC$ label, Nelson--Weiler use
\[
 \mathfrak{s}_{\xi,\Gamma}=\mathfrak{s}_\xi+\PD(\Gamma)
\]
and quote Taubes's isomorphism in the form
\[
 \ECH_*(Y,\xi,\Gamma;\F)
 \cong
 \HM^{-*}(Y,\mathfrak{s}_{\xi,\Gamma};\F);
\]
see \cite[Equation (7.6)]{NelsonWeiler} and
\cite{TaubesECHSWI,TaubesECHSWII,TaubesECHSWIII,TaubesECHSWIV,TaubesECHSWV}.
This fixes the plus sign in the $\SpinC$ shift.  With the convention above,
the isomorphism is relative-degree preserving as a map to $\HM^{-*}$ and
relative-degree reversing as a map to the unreindexed cohomological grading.
The relative $\Z$-grading is available by
\eqref{eq:spinc-c1-torsion}; compare \cite[Remark 1.2]{NelsonWeiler}.
\end{proof}

For our boundary,
\begin{equation}\label{eq:boundarydataFloer}
 g=28,
 \qquad
 n=36,
 \qquad
 \chi(\Sigma_g)=-54.
\end{equation}

The next proposition is the numerical input that distinguishes the two
possible orientations of the fiber.

\begin{proposition}[Relative-grading asymmetry]\label{prop:HMgradingasymmetry}
The groups
\begin{equation}\label{eq:two-HM-groups}
 \HM^{-*}(Y,\mathfrak{s}_\xi+\PD(\mathfrak f);\F),
 \qquad
 \HM^{-*}(Y,\mathfrak{s}_\xi-\PD(\mathfrak f);\F)
\end{equation}
are not isomorphic as relatively $\Z$-graded vector spaces.  After translating
the two relative gradings so that their unique lowest ECH classes have degree
zero, their graded dimensions agree in degrees $0,\ldots,15$ and differ by one
in degree $16$.
\end{proposition}

\begin{proof}
It suffices to compare the ECH gradings in
Theorem~\ref{thm:HMgradedformula}.
For the positive fiber class $r=1$, formula \eqref{eq:relative-degree} becomes
\begin{equation}\label{eq:Iplus}
 I_+(d,q)=36d^2-88d+q.
\end{equation}
The lower endpoints $q=0$ of the summands for $d=0,1,2,3$ are
\[
 0,\quad -52,\quad -32,\quad 60,
\]
and the quadratic expression increases thereafter.  Hence the least degree is
$-52$, attained uniquely at $(d,q)=(1,0)$ by the admissible monomial
$e_-^{37}$.  After translating by $52$, the $d=1$ summand begins in degree $0$,
the $d=2$ summand in degree $20$, and the $d=0$ summand in degree $52$.

For the negative fiber class, the equality $-\mathfrak f=(n-1)\mathfrak f$ gives the representative
$r=35$.  Formula \eqref{eq:relative-degree} becomes
\begin{equation}\label{eq:Iminus}
 I_-(d,q)=36d^2-20d+q.
\end{equation}
Its least degree is $0$, attained uniquely at $(d,q)=(0,0)$ by $e_-^{35}$.
The $d=1$ summand begins in degree $16$, and the $d=2$ summand begins in degree
$104$.

\begin{center}
\begin{tabular}{c|c|c|c}
\toprule
class & representative $r$ & normalized minimum & next new $d$-summand \\
\midrule
$+\mathfrak f$ & $1$ & $d=1$, $q=0$ & $d=2$ starts at $20$ \\
$-\mathfrak f$ & $35$ & $d=0$, $q=0$ & $d=1$ starts at $16$ \\
\bottomrule
\end{tabular}
\end{center}

For fixed total multiplicity $M$ and internal degree $q$, the dimension of the
corresponding subspace of $\mathcal A_M$ is
\begin{equation}\label{eq:monomial-count}
 C(M,q)=
 \sum_{\substack{0\leq j\leq\min(2g,q)\\
                  j\equiv q\; (\mathrm{mod}\,2)\\
                  M-(q+j)/2\geq0}}
 \binom{2g}{j}.
\end{equation}
Indeed, after choosing the $j$ distinct hyperbolic generators, the equations
$q=j+2m_+$ and $M=m_-+j+m_+$ determine $m_+$ and $m_-$ uniquely.  If $M\geq q$,
then $j\leq q$ gives $(q+j)/2\leq q\leq M$, so the final inequality is
automatic and $C(M,q)$ is independent of $M$.

For every $0\leq\ell<16$, normalized degree $\ell$ on the positive side comes
only from $d=1$, $q=\ell$, $M=37$, while on the negative side it comes only from
$d=0$, $q=\ell$, $M=35$.  Since $37,35\geq\ell$, the two dimensions are equal.
At degree $16$, the positive side has only the contribution
$(d,q,M)=(1,16,37)$.  The negative side has the contribution
$(d,q,M)=(0,16,35)$ of the same dimension and, in addition, the unique class
$(d,q,M)=(1,0,71)$.  Consequently,
\begin{equation}\label{eq:degree16}
 \dim_{\F}\ECH_{16}^{\mathrm{norm}}(Y,\xi,-\mathfrak f)
 =
 \dim_{\F}\ECH_{16}^{\mathrm{norm}}(Y,\xi,\mathfrak f)+1.
\end{equation}
Both relatively graded groups are bounded below, and their lowest-degree
subspaces are one-dimensional.  Hence any relatively graded isomorphism,
which is allowed to shift all degrees by an overall constant, must send the
lowest-degree subspace on one side to the lowest-degree subspace on the
other.  After normalizing both lowest degrees to zero, this overall shift is
therefore forced to be zero.  Such an isomorphism would consequently have
to preserve the dimension in every normalized degree.  Equation
\eqref{eq:degree16} shows that the dimensions in normalized degree $16$
differ by one, and hence no relatively graded isomorphism can exist.
\end{proof}
\begin{remark}[The normalized dimensions in degrees $0$ through $16$]
For $0\leq\ell\leq16$, put
\[
 b_\ell^{\pm}
 =\dim_{\F}\ECH_{\ell}^{\mathrm{norm}}
   (Y,\xi,\pm\mathfrak f;\F).
\]
For $0\leq\ell\leq15$, both dimensions equal
\[
 \sum_{\substack{0\leq j\leq\ell\\ j\equiv\ell\;({\rm mod}\,2)}}
 \binom{56}{j}.
\]
At degree $16$, the $-\mathfrak f$ side has the additional class from the
$d=1$, $q=0$ summand.  The resulting dimensions are as follows.
\begin{center}
\small
\renewcommand{\arraystretch}{1.06}
\begin{tabular}{r|r|r}
\toprule
normalized degree $\ell$ & $b_\ell^{+}$ & $b_\ell^{-}$ \\
\midrule
$0$ & $1$ & $1$ \\
$1$ & $56$ & $56$ \\
$2$ & $1{,}541$ & $1{,}541$ \\
$3$ & $27{,}776$ & $27{,}776$ \\
$4$ & $368{,}831$ & $368{,}831$ \\
$5$ & $3{,}847{,}592$ & $3{,}847{,}592$ \\
$6$ & $32{,}837{,}267$ & $32{,}837{,}267$ \\
$7$ & $235{,}764{,}992$ & $235{,}764{,}992$ \\
$8$ & $1{,}453{,}331{,}342$ & $1{,}453{,}331{,}342$ \\
$9$ & $7{,}811{,}733{,}392$ & $7{,}811{,}733{,}392$ \\
$10$ & $37{,}060{,}382{,}822$ & $37{,}060{,}382{,}822$ \\
$11$ & $156{,}713{,}948{,}672$ & $156{,}713{,}948{,}672$ \\
$12$ & $595{,}443{,}690{,}122$ & $595{,}443{,}690{,}122$ \\
$13$ & $2{,}046{,}626{,}681{,}072$ & $2{,}046{,}626{,}681{,}072$ \\
$14$ & $6{,}400{,}175{,}653{,}922$ & $6{,}400{,}175{,}653{,}922$ \\
$15$ & $18{,}299{,}876{,}179{,}712$ & $18{,}299{,}876{,}179{,}712$ \\
$16$ & $48{,}049{,}127{,}494{,}187$ & $48{,}049{,}127{,}494{,}188$ \\
\bottomrule
\end{tabular}
\end{center}

\end{remark}
\begin{theorem}[Positive fiber-class detection]\label{thm:fiberdetection}
Let $(Y,\xi_\pm)$ be the prequantization contact manifolds arising in
Corollary~\ref{cor:boundarynumbers}.  
Put
\[
 \mathfrak f_\pm=[t_\pm]\in H_1(Y;\Z).
\]
If
\[
 h:(Y,\xi_+)\longrightarrow(Y,\xi_-)
\]
is a coorientation-preserving contactomorphism, then
\[
 h_*[t_+]=[t_-]
\]
as oriented free-homotopy classes, and
\[
 h_*\mathfrak f_+=\mathfrak f_-
\]
in homology.
\end{theorem}

\begin{proof}

Recall the coorientation-preserving contactomorphisms
\[
 \kappa_\pm:(Y, \xi_\pm)\longrightarrow
 (Y_{28,-36},\xi_{\mathrm{BW}})
\]
chosen after Corollary~\ref{cor:boundarynumbers}; they carry the oriented
positive fiber classes $[t_\pm]$ to the standard positive fiber class
$[t]$.  Consider the diagram
\[
\begin{array}{ccc}
 (Y, \xi_+) & \xrightarrow{\ h\ } & (Y, \xi_-)\\
 \big\downarrow\rlap{$\scriptstyle\kappa_+$} &&
 \big\downarrow\rlap{$\scriptstyle\kappa_-$}\\
 (Y_{28,-36},\xi_{\mathrm{BW}}) &
 \xrightarrow{\ \kappa_-h\kappa_+^{-1}\ } &
 (Y_{28,-36},\xi_{\mathrm{BW}}).
\end{array}
\]
For each
$\sigma\in\{+,-\}$, we have isomorphisms
\[
 \kappa_\sigma^*:
 \HM^{-*}\!\left(Y_{28,-36},
   \mathfrak s_{\xi_{\mathrm{BW}}}\pm\PD(\mathfrak f);\F\right)
 \xrightarrow{\ \cong\ }
 \HM^{-*}\!\left(Y,
   \mathfrak s_{\xi_\sigma}\pm\PD(\mathfrak f_\sigma);\F\right).
\]
The contactomorphism $h$ induces the contravariant map
\[
 h^*: \widehat{HM}^{\bullet}(Y, \mathfrak s)
 \longrightarrow
 \widehat{HM}^{\bullet}(Y, h^*\mathfrak s).
\]
 Moreover, a coorientation-preserving
contactomorphism preserves the oriented contact plane field, and hence
\[
 h^*\mathfrak s_{\xi_-}=\mathfrak s_{\xi_+}.
\]
Choose a basepoint $y_+\in Y$ and set
\[
 y_-:=h(y_+).
\]
Then $h$ induces an isomorphism
\[
 h_*:\pi_1(Y,y_+)\longrightarrow\pi_1(Y,y_-).
\]
Let $t_\pm$ also denote the positively oriented fiber loops based at
$y_\pm$.  By Lemma~\ref{lem:center} (iv), the centers of the two fundamental
groups are
\[
 Z\bigl(\pi_1(Y,y_\pm)\bigr)=\langle t_\pm\rangle\cong\Z.
\]
Since a group isomorphism carries the center isomorphically onto the center,
$h_*(t_+)$ must be a generator of $\langle t_-\rangle$.  The only generators
of this infinite cyclic group are $t_-$ and $t_-^{-1}$; hence
\[
 h_*t_+=t_-^\varepsilon,
 \qquad
 \varepsilon\in\{+1,-1\}.
\]

Suppose $\varepsilon=-1$.  Then
$h_*\mathfrak f_+=-\mathfrak f_-$ and
\begin{equation}\label{eq:Floer-label-pullback}
 h^*\bigl(\mathfrak s_{\xi_-}-\PD(\mathfrak f_-)\bigr)
 =\mathfrak s_{\xi_+}+\PD(\mathfrak f_+).
\end{equation}

Diffeomorphism naturality
gives a relatively graded isomorphism
\[
 \HM^{-*}(Y,\mathfrak{s}_{\xi_+}+\PD(\mathfrak f_+);\F)
 \cong
 \HM^{-*}(Y,\mathfrak{s}_{\xi_-}-\PD(\mathfrak f_-);\F).
\]
Conjugating this isomorphism by the pullback maps $\kappa_+^*$ and
$\kappa_-^*$ displayed above gives a relatively graded isomorphism
\[
(\kappa_+^*)^{-1}\circ h^*\circ\kappa_-^*:
\HM^{-*}\!\left(
Y_{28,-36},
\mathfrak s_{\xi_{\mathrm{BW}}}-\PD(\mathfrak f);\F
\right)
\xrightarrow{\ \cong\ }
\HM^{-*}\!\left(
Y_{28,-36},
\mathfrak s_{\xi_{\mathrm{BW}}}+\PD(\mathfrak f);\F
\right).
\]
This contradicts Proposition~\ref{prop:HMgradingasymmetry}. Hence $\varepsilon=+1$.

\end{proof}

\section{Weinstein inequivalence}\label{sec:weinstein}

We now prove part~\textup{(v)} of the main theorem.  This is the only
non-equivalence argument that uses the monopole Floer and ECH calculation of
Section~\ref{sec:monopole}: a Weinstein homotopy determines an endpoint
contactomorphism, and Theorem~\ref{thm:fiberdetection} is needed to recover the
orientation of its circle fiber.

\begin{theorem}[Weinstein inequivalence]\label{thm:nonweinstein}
Let $(X,Y,\xi_\pm)$ and the normalized Stein data be as above, and let
\[
 \W_+=(X,\Lambda_+,\varphi_+),
 \qquad
 \W_-=(X,\Lambda_-,\varphi_-)
\]
be Weinstein structures whose cooriented boundary contact structures are
$\xi_+$ and $\xi_-$, respectively.  
Here $\Lambda_\pm$ denote arbitrary Weinstein Liouville forms, not necessarily equal to the $\lambda_\pm=d^c_{J_\pm}\rho_\pm$ associated with the normalized Stein data.
Then there is no diffeomorphism $F:X\to X$ for
which $\W_+$ is Weinstein homotopic to $F^*\W_-$.  In particular, this applies
to the Weinstein structures associated with the normalized Stein data
$(J_+,\rho_+)$ and $(J_-,\rho_-)$.
\end{theorem}

\begin{proof}
Suppose that such a diffeomorphism $F$ and a Weinstein homotopy exist.  If $F$
reverses the fixed orientation of $X$, then $F^*d\Lambda_-$ induces the
orientation opposite to that of $d\Lambda_+$, so no path of symplectic forms,
and hence no Weinstein homotopy, can join them.  We may therefore assume that
$F$ preserves orientation.  On $Y=\partial X$, the homotopy induces a path
$\xi_s$ of cooriented contact structures with
\[
 \xi_0=\xi_+,
 \qquad
 \xi_1=(F|_Y)^*\xi_-.
\]
Let $\psi_s:Y\to Y$ be the Gray isotopy satisfying
$\psi_s^*\xi_s=\xi_+$.  Then
\[
 h=F|_Y\circ\psi_1:(Y,\xi_+)\longrightarrow(Y,\xi_-)
\]
is a coorientation-preserving contactomorphism isotopic to $F|_Y$.
Theorem~\ref{thm:fiberdetection} gives
\[
 h_*[t_+]=[t_-].
\]
Since $u_\pm=t_\pm^{-1}$, the map also preserves the oriented disk-boundary
fibers $u_\pm$.  Proposition~\ref{prop:capextension} supplies an isotopy from
$h$ to a fiber-preserving homeomorphism $h_0$ and an orientation-preserving
cap map
\[
 \widehat h_0:\nu(D_+)\longrightarrow\nu(D_-)
\]
over the normal disk-bundle caps, carrying the oriented zero section $D_+$ to
$D_-$.  Concatenate the Gray isotopy from $F|_Y$ to $h$ with this fiber-preserving
isotopy and apply Lemma~\ref{lem:collar-adjustment} to the underlying
homeomorphism of $F$.  We obtain an orientation-preserving homeomorphism
$F_0:X\to X$ which agrees with $F$ off a collar and has boundary value $h_0$.
The full extension over the normal disk-bundle cap is explicitly
\[
 \overline h_0
 =F_0\cup_{h_0}\widehat h_0:
 X\cup_Y\nu(D_+)\longrightarrow X\cup_Y\nu(D_-).
\]
After identifying both closed manifolds with $S$, this is a self-homeomorphism
carrying the oriented zero section $D_+$ to $D_-$.  Exactly as in
\eqref{eq:capclassrelation}, one obtains
\[
 \overline h_0^{*}(-2K)=2K,
\]
whereas Proposition~\ref{prop:KKsurface} gives
$\overline h_0^{*}K=K$.  This would force $4K=0$, contradicting
$K^2=9$, again.
\end{proof}

\begin{remark}[Why Floer theory input is needed here]
\label{rem:why-floer-v}
In the proof of Theorem~\ref{thm:nonweinstein}, the map $F_0$ obtained after
the collar adjustment is only an orientation-preserving homeomorphism; it is
not a symplectomorphism.  Consequently Lemma~\ref{lem:symp-characteristic}
cannot be applied to $F_0$.  More fundamentally, before the adjustment the
Gray-stability map $h$ is an arbitrary coorientation-preserving
contactomorphism and need not preserve the characteristic foliation of the
selected normalized forms $\alpha_\pm$.

This is precisely the difference between parts~\textup{(iv)} and
\textup{(v)}.  In part~\textup{(iv)}, the original map $F$ is a
symplectomorphism, so Lemma~\ref{lem:symp-characteristic} yields
$F_*[t_+]=[t_-]$ before any collar modification; no Floer theory is needed.
In part~\textup{(v)}, that direct argument is unavailable, and
Theorem~\ref{thm:fiberdetection} supplies the fiber-class preservation before
one passes to $F_0$ and performs the cap extension.
\end{remark}

\begin{proof}[Proof of Theorem~\ref{thm:main}]
Use the conjugate Stein domains of Section~\ref{sec:stein} and the normalized
defining functions of Proposition~\ref{prop:normalized-data}.  Part
\textup{(i)} is Proposition~\ref{prop:normalized-data}\textup{(iii)}.  The
equality of first Chern classes in part~\textup{(ii)} is
Proposition~\ref{prop:c1equal}, and part~\textup{(iii)} is
Proposition~\ref{prop:spincdifferent}.  Theorem~\ref{thm:nonsymp} proves part~\textup{(iv)}, and
Theorem~\ref{thm:nonweinstein} proves part~\textup{(v)}.  The numerical
boundary description is Corollary~\ref{cor:boundarynumbers}.  Equality of the
symplectic volumes is
Proposition~\ref{prop:normalized-data}\textup{(v)}.

Any Morse perturbation needed to obtain a Weinstein function is supported away
from the boundary collar and therefore preserves the strict equality of the
boundary contact forms.
\end{proof}

\begin{proof}[Proof of Corollary~\ref{cor:k3}]
Regard both domains as fillings of the cooriented contact manifold
$(Y,\xi_+)$: use the identity for the positive boundary and
$\tau^{-1}:(Y,\xi_-)\to(Y,\xi_+)$ for the negative boundary.  The identity on
$X$ matches their first Chern classes, while
Theorems~\ref{thm:nonsymp} and~\ref{thm:nonweinstein} exclude all symplectic
and Weinstein equivalences.  The strict boundary equality in Theorem~\ref{thm:main} is precisely the
boundary-form condition in the Main problem.  No compatibility between the
identity on $X$ and the boundary contactomorphism $\tau$ is required.
\end{proof}

\section{Further remarks and questions}\label{sec:further}

\begin{remark}[Relation with Heegaard Floer theory]
The Lisca--Mati\'c theorem and Plamenevskaya's Heegaard Floer refinement are
often summarized as saying that nonisomorphic Stein $\SpinC$ structures give
distinct contact invariants; see \cite{KarakurtObaUkida}.  There is no
contradiction with the contactomorphism $\tau$ because
Remark~\ref{rem:marking-not-extend} shows that no boundary contactomorphism
between the two examples extends over the filling.
\end{remark}

\begin{remark}[No boundary contactomorphism extends]
\label{rem:marking-not-extend}
More strongly than the statement for the selected strict map $\tau$, no
coorientation-preserving contactomorphism from $(Y,\xi_+)$ to $(Y,\xi_-)$ is
isotopic to the boundary value of a diffeomorphism $X\to X$.  Such
contactomorphisms do exist: for example, the strict map $\tau$ constructed in
Proposition~\ref{prop:normalized-data} is one.

Suppose, to the contrary, that there are a coorientation-preserving
contactomorphism
\[
 h:(Y,\xi_+)\longrightarrow(Y,\xi_-)
\]
and a diffeomorphism $F:X\to X$ such that $F|_Y$ is isotopic to $h$.
Since $h$ preserves the contact orientation, it preserves the boundary
orientation.  With the outward-normal-first convention, $F$ must therefore
preserve the orientation of $X$.  Isotopy extension allows us to change the
boundary value of $F$ to $h$.  Theorem~\ref{thm:fiberdetection} gives
$h_*[t_+]=[t_-]$, and Proposition~\ref{prop:capextension} then extends $h$
over the normal disk-bundle cap.  Gluing this cap map to $F$ would produce a
self-homeomorphism $G:S\to S$ satisfying both $G^*K=K$ and
$G^*(-2K)=2K$, which is impossible.

Thus the identity used for the Chern-class comparison and every
contactomorphism between the two boundary contact manifolds necessarily
determine different, incompatible markings.  The boundary mapping classes
which extend over $X$ form a subgroup of the boundary mapping-class group.
Consequently, composing a nonextendable boundary class on either side with
extendable classes cannot make it extendable: otherwise the subgroup property
would imply that the original class extended.  Thus the strict boundary
identification in the Main problem is necessarily unmarked in this example.
\end{remark}

\subsection*{A marked version}

\begin{question}[Marked Main problem]
\label{ques:marked-main}
Fix a closed cooriented contact $3$--manifold $(Y,\xi)$ and a contact
form $\alpha$ with $\ker\alpha=\xi$.  Do there exist compact connected
Stein surfaces $(X_i,J_i,\rho_i)$, contactomorphisms
\[
 \iota_i:
 (\partial X_i,\ker(\lambda_i|_{\partial X_i}))
 \longrightarrow(Y,\xi),
 \qquad
 \lambda_i=d_{J_i}^c\rho_i,
\]
and a diffeomorphism $f:X_1\to X_2$ such that
\[
 f^*c_1(TX_2,J_2)=c_1(TX_1,J_1),
 \qquad
 \iota_i^*\alpha=\lambda_i|_{\partial X_i}
 \quad(i=1,2),
\]
and such that
\[
 \iota_2\circ f|_{\partial X_1}
\]
is isotopic to $\iota_1$ through coorientation-preserving
contactomorphisms, but no marking-compatible symplectomorphism exists?
Can one moreover require that no marking-compatible diffeomorphism makes
the associated Weinstein structures Weinstein homotopic?
\end{question}

Here a diffeomorphism
\[
 \Phi:X_1\longrightarrow X_2
\]
is \emph{marking-compatible} if
\[
 \iota_2\circ\Phi|_{\partial X_1}
\]
is isotopic to $\iota_1$ through coorientation-preserving
contactomorphisms.

\begin{remark}[Relation with Wang's example]
\label{rem:Wang}
There is a related four-dimensional phenomenon due to Luya Wang.  In the
compact-core formulation used here, Wang proved that the space of symplectic
forms on $S^1\times D^3$ which agree with a fixed standard symplectic form
$\omega_{\mathrm{std}}$ on a collar neighborhood of the boundary has
infinitely many connected components
\cite[Theorem~1]{WangSymplecticComponents}.  More precisely, her examples
are obtained in the form
\[
 \omega_i=f_i^*\omega_{\mathrm{std}},
 \qquad
 f_i\in\Diff_c(S^1\times D^3),
\]
where each $f_i$ is the identity on a collar neighborhood of the boundary,
and the forms are pairwise non-isotopic through symplectic forms which
continue to agree with $\omega_{\mathrm{std}}$ near the boundary.

Wang's construction also gives a corresponding statement in the Stein
category.  Start with a standard Stein datum
\[
 (J_{\mathrm{std}},\phi_{\mathrm{std}})
\]
whose associated symplectic form is
\[
 \omega_{\mathrm{std}}
 =
 dd^c_{J_{\mathrm{std}}}\phi_{\mathrm{std}},
\]
and put
\[
 J_i=f_i^*J_{\mathrm{std}},
 \qquad
 \phi_i=f_i^*\phi_{\mathrm{std}}.
\]
Then
\[
 dd^c_{J_i}\phi_i
 =
 f_i^*
 \bigl(dd^c_{J_{\mathrm{std}}}\phi_{\mathrm{std}}\bigr)
 =
 \omega_i.
\]
For any fixed pair $i,j$, the Stein data $(J_i,\phi_i)$ and
$(J_j,\phi_j)$ agree with the standard Stein datum on a common collar
neighborhood of the boundary.  In particular, their boundary contact forms,
and not merely their contact structures up to isotopy, are identical.

The two Stein data cannot be joined by a Stein homotopy fixed on a collar
neighborhood of the boundary.  Indeed, the associated exact symplectic forms
would then give a path from $\omega_i$ to $\omega_j$ through symplectic forms
which continue to agree with $\omega_{\mathrm{std}}$ near the boundary,
contradicting Wang's theorem.  Thus Wang's construction shows that even a
fixed boundary contact germ does not determine the Stein-homotopy class when
the underlying smooth marking is kept fixed.

This phenomenon should be distinguished from
Question~\ref{ques:marked-main}.  Wang's examples do \emph{not} satisfy its
non-symplectomorphism requirement.  Indeed, for any $i,j$ the diffeomorphism
\[
 \Phi_{ij}:=f_j^{-1}\circ f_i
\]
is the identity near the boundary and satisfies
\[
 \Phi_{ij}^*\omega_j=\omega_i.
\]
Hence $\Phi_{ij}$ is a marking-compatible symplectomorphism.  In fact, it
identifies the Stein data themselves:
\[
 \Phi_{ij}^*J_j=J_i,
 \qquad
 \Phi_{ij}^*\phi_j=\phi_i.
\]
Consequently the associated Weinstein structures become identical after
pullback by a marking-compatible diffeomorphism.  Thus Wang detects
nontrivial connected components of the space of structures with the smooth
marking fixed, whereas Question~\ref{ques:marked-main} asks for inequivalence
even after allowing diffeomorphisms compatible with the boundary marking.

The example constructed in the present paper does not answer
Question~\ref{ques:marked-main} for a different reason: although its
inequivalence survives arbitrary reparametrization of the filling, the chosen
strict boundary identification is not isotopic to the boundary restriction of
any diffeomorphism of the filling; see
Remark~\ref{rem:marking-not-extend}.  Thus the two constructions exhibit two
different forms of boundary nonuniqueness, while the marked problem asks for
both features simultaneously.

In this compact-core formulation, Wang's construction gives the first
nontrivial connected component of the space of symplectic forms agreeing with
a fixed standard form near the boundary in real dimension four.  Thus, in
this precise fixed-marking isotopy sense, no earlier four-dimensional example
of this phenomenon was known.
\end{remark}

\subsection*{A four-dimensional refinement}

\begin{question}
Can one obtain an analogous four-dimensional example whose Stein filling is
simply connected?
\end{question}

\subsection*{Questions in the remaining dimensions}

Proposition~\ref{prop:dimension-two}, Theorem~\ref{thm:main}, and
Proposition~\ref{prop:higher-main} give, respectively, a negative answer in
real dimension two and affirmative answers in real dimension four and in every
real dimension $4k+2\geq10$.  This leaves the following dimension-specific
questions.

\begin{question}[Real dimension six]
Does the Main problem have an affirmative solution for compact connected Stein
threefolds, including the non-Weinstein-homotopic clause?
\end{question}

\begin{question}[Real dimensions divisible by four]
For every $k\geq2$, does the Main problem have an affirmative solution in real
dimension $4k$, including the non-Weinstein-homotopic clause?
\end{question}

\section*{Acknowledgements}
The author thanks ChatGPT for extensive discussions.  This project grew out of
a dialogue between the author and ChatGPT motivated by the author's attempt to
understand the theorem of Lisca and Mati\'c.  The mathematical statements and
proofs in the final manuscript were checked and are the responsibility of the
author.
This work was supported by JSPS KAKENHI Grant Number JP26K16990 and by the Start-up Fund of the Kavli Institute for the Physics and Mathematics of the Universe (Kavli IPMU), the University of Tokyo.	

\begingroup
\sloppy
\hbadness=10000
\bibliographystyle{alpha}
\bibliography{Steinfinal}
\endgroup

\end{document}